\documentclass[a4paper,12pt,reqno]{amsart}

\usepackage{amsmath}
\usepackage{amssymb}
\usepackage{amsfonts}
\usepackage{graphicx}
\usepackage{mathtools} 
\usepackage[colorlinks]{hyperref}
\renewcommand\eqref[1]{(\ref{#1})} 

\graphicspath{ {images/} }
\providecommand{\Real}{\mathop{\rm Re}\nolimits}%
\title[Variable-coefficient Prabhakar differential equations]{Prabhakar-type linear differential equations with variable coefficients}

\author[A. Fernandez]{Arran Fernandez}
\address{
	Arran Fernandez:
	\endgraf
	Department of Mathematics
	\endgraf
	Eastern Mediterranean University
	\endgraf
	Northern Cyprus, via Mersin-10, Turkey
	\endgraf
	{\it E-mail address:} {\rm arran.fernandez@emu.edu.tr}}

\author[J. E. Restrepo]{Joel E. Restrepo}
\address{
	Joel E. Restrepo:
	\endgraf
	Department of Mathematics 
	\endgraf
	Nazarbayev University
	\endgraf
	Kazakhstan
	\endgraf
	and
	\endgraf
	Department of Mathematics: Analysis, Logic and Discrete Mathematics
	\endgraf
	Ghent University, Krijgslaan 281, Building S8, B 9000 Ghent
	\endgraf
	Belgium
	\endgraf
	{\it E-mail address:} {\rm cocojoel89@yahoo.es;\,joel.restrepo@ugent.be}}

\author[D. Suragan]{Durvudkhan Suragan}
\address{
	Durvudkhan Suragan:
	\endgraf
	Department of Mathematics
	\endgraf
	Nazarbayev University
	\endgraf
	Kazakhstan
	\endgraf
	{\it E-mail address:} {\rm durvudkhan.suragan@nu.edu.kz}}


\newtheoremstyle{theorem}
{10pt}          
{10pt}  
{\sl}  
{\parindent}     
{\bf}  
{. }    
{ }    
{}     
\theoremstyle{theorem}

\numberwithin{equation}{section}
\theoremstyle{plain}
\newtheorem{thm}{Theorem}[section]

\newtheorem{lem}[thm]{Lemma}

\theoremstyle{definition}
\newtheorem{defn}[thm]{Definition}
\newtheorem{rem}[thm]{Remark}

\newtheoremstyle{defi}
{10pt}          
{10pt}  
{\rm}  
{\parindent}     
{\bf}  
{. }    
{ }    
{}     
\theoremstyle{defi}

\begin{document}
 	\begin{abstract}
Linear differential equations with variable coefficients and Prabhakar-type operators featuring Mittag-Leffler kernels are solved. In each case, the unique solution is constructed explicitly as a convergent infinite series involving compositions of Prabhakar fractional integrals. We also extend these results to Prabhakar operators with respect to functions. As an important illustrative example, we consider the case of constant coefficients, and give the solutions in a more closed form by using multivariate Mittag-Leffler functions.
	\end{abstract}
	\maketitle
	\tableofcontents

\section{Introduction}

Fractional differential equations (FDEs) are widely studied, both from the pure mathematical viewpoint \cite{kilbas,podlubny,samko} and due to their applications in assorted fields of science and engineering \cite{hilfer,sun-etal}. The simple case of linear ordinary FDEs with constant coefficients has been thoroughly studied in classical textbooks such as \cite{kilbas,miller}, but many other FDE problems are still providing challenges to mathematical researchers.

Explicit solutions have been constructed for several classes of linear FDEs with variable coefficients. Different approaches have been considered to obtain representations of solutions for such equations, including Green's functions \cite{RL}, the Banach fixed point theorem \cite{first,analitical}, power series methods \cite{AML,kilbasalpha,vcapl}, and Volterra integral equations \cite{vcserbia1,vcserbia2}. The tools used in \cite{first,RL,analitical} yielded representations of the solutions by uniformly convergent infinite series involving nested compositions of Riemann--Liouville fractional integrals. This is relatively easy to handle compared with other representations where sometimes reproducing kernels are involved, and the nested fractional integrals can even be eliminated to obtain a formula more suitable for numerical calculation \cite{FRS}. The starting point of the method in these papers was to exchange the original fractional differential equation for an equivalent integral equation, a very useful technique which, to the best of our knowledge, was first used for FDEs by Pitcher and Sewel in \cite{AMS-1938}.

Recently, the study of explicit solutions of FDEs with variable coefficients has been growing in attention and opening new directions of investigation and application. After the works \cite{RL,analitical} where the problem was solved in the classical settings of Riemann--Liouville and Caputo fractional derivatives, several other papers have extended the same methodology to other types of fractional derivatives, such as Caputo derivatives with respect to functions and derivatives with non-singular Mittag-Leffler kernels \cite{RRS,FRS:AB}. This method has also been applied to partial differential equations \cite{RSade}, and in the investigation of inverse fractional Cauchy problems of wave and heat type, it was also used to define a new class of time-fractional Dirac type operators with time-variable coefficients and with applications in fractional Clifford analysis \cite{BRS,RRSdirac}. Such operators of fractional Dirac type lead to the consideration of a wide range of fractional Cauchy problems, whose solutions were given explicitly. 

In this paper, we study the explicit solutions of variable-coefficient FDEs in the setting of Prabhakar fractional derivatives. The origins of Prabhakar fractional calculus lie in the fractional integral operator introduced in \cite{Prab1971}, which was more deeply studied in \cite{generalizedfc} and extended to fractional derivatives in \cite{prabcap}. Recently, Prabhakar fractional calculus has been intensively studied both for its pure mathematical properties \cite{fernandez-baleanu,giusti-etal} and for its assorted applications \cite{garrappa-maione,tomovski-dubbeldam-korbel}, so Prabhakar fractional differential equations have become a topic of interest \cite{RS:MMAS}. For this reason, we have conducted the current research into fractional differential equations with variable coefficients and Prabhakar derivatives, constructing explicit solutions using the methodology of \cite{analitical}.

The structure of the paper is given as follows. In Section \ref{preliPrabFDE}, we collect all necessary definitions and preliminary results on Prabhakar fractional calculus, as well as Prabhakar operators with respect to functions. Section \ref{mainPrabFDE} is devoted to the main results: proving existence and uniqueness for the considered Prabhakar-type linear differential equation with variable coefficients, constructing explicitly a canonical set of solutions, and finally finding the explicit form of the unique solution, both for the Prabhakar-type differential equation and also for its generalisation using Prabhakar operators with respect to functions. In Section \ref{FDEPrabconstcoe}, as an illustrative example of our general results, we write explicit solutions for the general linear Prabhakar-type FDE with constant coefficients, by using the multivariate Mittag-Leffler function.

\section{Preliminaries}\label{preliPrabFDE}

Let us recall the main definitions and auxiliary results that will be used in this paper.

\subsection{Prabhakar fractional calculus}

Before introducing the operators of Prabhakar fractional calculus, we need to recall the three-parameter Mittag-Leffler function $E^{\theta}_{\alpha,\beta}$, which was introduced and studied by Prabhakar in \cite{Prab1971}: 
\[
E^{\theta}_{\alpha,\beta}(z)=\sum_{n=0}^{\infty}\frac{(\theta)_n}{\Gamma(\alpha n+\beta)}\cdot\frac{z^n}{n!},\quad z,\beta,\alpha,\theta\in\mathbb{C},\textrm{Re}\,\alpha>0,
\]
where $\Gamma(\cdot)$ is the Gamma function and $(\theta)_n$ is the Pochhammer symbol \cite[\S2.1.1]{pocha}, i.e. $(\theta)_n=\frac{\Gamma(\theta+n)}{\Gamma(\theta)}$ or
\[
(\theta)_0=1,\quad (\theta)_n=\theta(\theta+1)\cdots(\theta+n-1)\quad (n=1,2,\ldots).
\]
For $\theta=1$, we obtain the two-parameter Mittag-Leffler function $E_{\alpha,\beta}$, namely
\[
E_{\alpha,\beta}(z)=\sum_{n=0}^{\infty}\frac{z^n}{\Gamma(\alpha n+\beta)},\quad z,\beta,\alpha\in\mathbb{C},\textrm{Re}\,\alpha>0.
\]
For $\beta=\theta=1$, we obtain the classical Mittag-Leffler function  $E_{\alpha}(z)=E_{\alpha,1}(z)$. For more details of various types of the Mittag-Leffler function, see e.g. the book \cite{mittag}.

Briefly, we discuss the convergence of the above series. Applying the ratio test to $c_n=\frac{(\theta)_n}{\Gamma(\alpha k+\beta)}\frac{z^n}{n!}$ and using Stirling's approximation \cite[1.18(4)]{pocha}, we have
\begin{align*}
\left|\frac{c_{n+1}}{c_n}\right|&=\left|\frac{\frac{(\theta)_{n+1}}{\Gamma(\alpha(n+1)+\beta)}\frac{z^{n+1}}{(n+1)!}}{\frac{(\theta)_n}{\Gamma(\alpha n+\beta)}\frac{z^n}{n!}}\right|=|z|\frac{|\theta+n|}{n+1}\frac{|\Gamma(\alpha n+\beta)|}{|\Gamma(\alpha n+\beta+\alpha)|} \\
&\sim |z|\frac{|\theta+n|}{n+1}\frac1{|\alpha n+\beta|^{\Real \,\alpha}}\to 0,\quad n\to\infty,
\end{align*}
and we see why the assumption $\Real (\alpha)>0$ is necessary for the definition.

We now recall the Prabhakar integral operator, which is defined by
\begin{equation}\label{IPrab}
\left(\prescript{}{a}{\mathbb{I}}_{\alpha,\beta,\omega}^{\theta}f\right)(t)=\int_a^t (t-s)^{\beta-1}E^{\theta}_{\alpha,\beta}(\omega(t-s)^{\alpha})f(s)\,\mathrm{d}s,
\end{equation}
where $\alpha,\beta,\theta,\omega\in\mathbb{C}$ with $\Real (\alpha)>0$ and $\Real (\beta)>0$. This operator is bounded for functions $f\in L^1(a,b)$ for any $b>a$; for more details, see \cite[Theorems 4,5]{generalizedfc}. Note that for $\theta=0$, $\prescript{}{a}{\mathbb{I}}_{\alpha,\beta,\omega}^{0}$ coincides with the Riemann--Liouville fractional integral of order $\beta$ \cite[Sections 2.3 and 2.4]{samko}:
\begin{equation}\label{fraci}
\prescript{RL}{a}I^{\beta}f(t)=\frac1{\Gamma(\beta)}\int_a^t (t-s)^{\beta-1}f(s)\,\mathrm{d}s,\quad \beta\in\mathbb{C},\quad\Real (\beta)>0.
\end{equation}
Two important properties of the Prabhakar operator are its semigroup property (in the parameters $\beta,\theta$) and its series formula, which were proved in \cite{generalizedfc} and \cite{fernandez-baleanu-srivastava} respectively. These are:
\begin{align}
\prescript{}{a}{\mathbb{I}}_{\alpha,\beta_1,\omega}^{\theta_1}\circ\prescript{}{a}{\mathbb{I}}_{\alpha,\beta_2,\omega}^{\theta_2}=\prescript{}{a}{\mathbb{I}}_{\alpha,\beta_1+\beta_2,\omega}^{\theta_1+\theta_2},\quad\Real (\alpha)>0,\Real (\beta_i)>0,i=1,2; \label{PI:semi} \\
\left(\prescript{}{a}{\mathbb{I}}_{\alpha,\beta,\omega}^{\theta}f\right)(t)=\sum_{n=0}^{\infty}\frac{(\theta)_n\omega^n}{n!}\prescript{RL}{a}I^{\alpha n+\beta}f(t),\quad\Real (\alpha)>0,\Real (\beta)>0. \label{PI:series}
\end{align}
Thanks to all of the above identities and relations, the Prabhakar integral operator \eqref{IPrab} is considered \cite{fernandez-baleanu,generalizedfc} as a generalised fractional integral operator, giving rise to a type of fractional calculus involving Mittag-Leffler kernels. It is a complete model of fractional calculus including fractional derivatives as well as integrals, as we shall see in the following statements. Firstly we recall the space $AC^n(a,b)$ ($n\in\mathbb{N}$), which is the set of real-valued functions $f$ whose derivatives exist up to order $n-1$ on $(a,b)$ and such that $f^{(n-1)}$ is an absolutely continuous function.

The Prabhakar derivative of Riemann--Liouville type is defined \cite{prabcap} by 
\begin{align}\label{DPrabRL}
\left(\prescript{RL}{a}{\mathbb{D}}_{\alpha,\beta,\omega}^{\theta}f\right)(t)&=\frac{\mathrm{d}^m}{\mathrm{d}t^m}\left(\prescript{}{a}{\mathbb{I}}_{\alpha,m-\beta,\omega}^{-\theta}f(t)\right) \nonumber\\
&=\frac{\mathrm{d}^m}{\mathrm{d}t^m}\int_a^t (t-s)^{m-\beta-1}E^{-\theta}_{\alpha,m-\beta}(\omega(t-s)^{\alpha})f(s)\,\mathrm{d}s,
\end{align}
where $\alpha,\beta,\theta,\omega\in\mathbb{C}$ with $\Real (\alpha)>0$, $\Real (\beta)\geqslant0$, and $m=\lfloor \Real \,\beta\rfloor+1$ (where $\lfloor\cdot\rfloor$ is the floor function) and $f\in AC^m(a,b)$.

The following inversion result for Prabhakar integrals and derivatives follows immediately from the semigroup property \eqref{PI:semi} and the classical fundamental theorem of calculus:
\begin{equation} \label{thm2.5PrabFDE}
\prescript{RL}{a}{\mathbb{D}}_{\alpha,\beta_2,\omega}^{\theta_2}\circ\prescript{}{a}{\mathbb{I}}_{\alpha,\beta_1,\omega}^{\theta_1}=
\begin{cases}
\prescript{}{a}{\mathbb{I}}_{\alpha,\beta_1-\beta_2,\omega}^{\theta_1-\theta_2},&\quad\Real (\beta_1)>\Real (\beta_2)\geqslant0; \\\\
\prescript{RL}{a}{\mathbb{D}}_{\alpha,\beta_2-\beta_1,\omega}^{\theta_2-\theta_1},&\quad\Real (\beta_2)\geqslant\Real (\beta_1)>0,
\end{cases}
\end{equation}
where $\alpha,\beta_i,\theta_i,\omega\in\mathbb{C}$ such that $\Real (\alpha)>0$ and $\Real (\beta_i)>0$ for $i=1,2$. In particular, for $\beta,\theta\in\mathbb{C}$ such that $\Real (\beta)>0$, we have
\[
\prescript{RL}{a}{\mathbb{D}}_{\alpha,\beta,\omega}^{\theta}\prescript{}{a}{\mathbb{I}}_{\alpha,\beta,\omega}^{\theta}f(t)=f(t),\quad f\in C[a,b].
\]

The Prabhakar derivative of Caputo type, sometimes also called the regularised Prabhakar derivative, is usually defined \cite{prabcap} by
\begin{align}
\left(\prescript{C}{a}{\mathbb{D}}_{\alpha,\beta,\omega}^{\theta}f\right)(t)&=\prescript{}{a}{\mathbb{I}}_{\alpha,m-\beta,\omega}^{-\theta}\left(\frac{\mathrm{d}^m}{\mathrm{d}t^m}f(t)\right) \nonumber\\
&=\int_a^t (t-s)^{m-\beta-1}E^{-\theta}_{\alpha,m-\beta}(\omega(t-s)^{\alpha})f^{(m)}(s)\,\mathrm{d}s, \label{DPrab}
\end{align}
where $\alpha,\beta,\theta,\omega\in\mathbb{C}$ with $\Real (\alpha)>0$, $\Real (\beta)\geqslant0$, and $m=\lfloor\Real \beta\rfloor+1$, and $f\in AC^m(a,b)$. Note that $f\in AC^m[a,b]$ is enough for \eqref{DPrab} to be well-defined, since this guarantees $f^{(m)}$ exists almost everywhere and is in $L^1[a,b]$, therefore the fractional integral of $f^{(m)}$ exists; we do not need stronger conditions such as $f\in C^m[a,b]$ for the existence of the Caputo-type derivative. Boundedness of the operator $\prescript{C}{a}{\mathbb{D}}_{\alpha,\beta,\omega}^{\theta}$ is discussed in \cite[Theorem 4]{polito}. For $\theta=0$, this operator coincides with the original Caputo fractional derivative.

We also have the following alternative formula for the Caputo--Prabhakar derivative, which is equivalent to \eqref{DPrab} for any function $f\in AC^m(a,b)$:
\begin{equation}\label{alternativePrabh}
\left(\prescript{C}{a}{\mathbb{D}}_{\alpha,\beta,\omega}^{\theta}f\right)(t)=\prescript{RL}{a}{\mathbb{D}}_{\alpha,\beta,\omega}^{\theta}\left[f(t)-\sum_{j=0}^{m-1}\frac{f^{(j)}(a)}{j!}(t-a)^{j}\right],
\end{equation}
where $\alpha,\beta,\theta,\omega\in\mathbb{C}$ with $\Real (\alpha)>0$, $\Real (\beta)\geqslant0$, and $m=\lfloor\Real\beta\rfloor+1$. The equivalence of \eqref{DPrab} and \eqref{alternativePrabh} was proved in \cite[Proposition 4.1]{prabcap}. In this paper, we shall use them both interchangeably. 

The Prabhakar derivatives, of both Riemann--Liouville and Caputo type, have series formulae analogous to \eqref{PI:series}, namely:
\begin{align}
\left(\prescript{RL}{a}{\mathbb{D}}_{\alpha,\beta,\omega}^{\theta}f\right)(t)=\sum_{n=0}^{\infty}\frac{(-\theta)_n\omega^n}{n!}\prescript{RL}{a}I^{\alpha n-\beta}f(t),\quad\Real (\alpha)>0,\Real (\beta)\geqslant0, \label{PR:series}\\
\left(\prescript{C}{a}{\mathbb{D}}_{\alpha,\beta,\omega}^{\theta}f\right)(t)=\sum_{n=0}^{\infty}\frac{(-\theta)_n\omega^n}{n!}\prescript{RL}{a}I^{\alpha n+m-\beta}f^{(m)}(t),\quad\Real (\alpha)>0,\Real (\beta)\geqslant0, \label{PC:series}
\end{align}
where in \eqref{PR:series} we use the analytic continuation of the Riemann--Liouville integral (called the Riemann--Liouville derivative) for the finitely many cases where $\Real (\alpha n-\beta)<0$. Note that the first term of the series in \eqref{PC:series} is precisely the classical Caputo derivative to order $\beta$ of $f$, defined by
\[
\prescript{C}{a}D^{\beta}f(t)=\prescript{RL}{a}I^{m-\beta}\left(\frac{\mathrm{d}^m}{\mathrm{d}t^m}f(t)\right)=\frac{1}{\Gamma(m-\beta)}\int_a^t (t-s)^{m-\beta-1}f^{(m)}(s)\,\mathrm{d}s,
\]
where $m:=\lfloor\Real \beta\rfloor+1$ as usual.

\begin{lem}\label{importantproPrabFDE}
If $\alpha,\beta,\theta,\omega\in\mathbb{C}$ with $\Real (\alpha)>0$, $\Real (\beta)>0$, and $f\in C[a,b]$, then the following statements hold:
\begin{enumerate}
\item $\left(\prescript{}{a}{\mathbb{I}}_{\alpha,\beta,\omega}^{\theta}f\right)(t)$ is a continuous function on $[a,b]$.
\item $\displaystyle\lim_{t\to a+}\left(\prescript{}{a}{\mathbb{I}}_{\alpha,\beta,\omega}^{\theta}f\right)(t)=0$.
\item If $\beta',\theta'\in\mathbb{C}$ with $\Real (\beta)>\Real (\beta')\geqslant0$, then
\[
\prescript{C}{a}{\mathbb{D}}_{\alpha,\beta',\omega}^{\theta'}\circ\prescript{}{a}{\mathbb{I}}_{\alpha,\beta,\omega}^{\theta}f(t)=\prescript{}{a}{\mathbb{I}}_{\alpha,\beta-\beta',\omega}^{\theta-\theta'}f(t).
\]
In particular, letting $\beta'\to\beta$ and $\theta'=\theta$, we have 
\[
\prescript{C}{a}{\mathbb{D}}_{\alpha,\beta,\omega}^{\theta}\circ\prescript{}{a}{\mathbb{I}}_{\alpha,\beta,\omega}^{\theta}f(t)=f(t).
\]
\end{enumerate}
\end{lem}

\begin{proof}
The first statement follows by \cite[Theorem 5]{generalizedfc}. The second statement is an application of the mean value theorem for integrals; note that the continuity of $f$ on the closed interval $[a,b]$ is vital for this.

Let us now prove the third statement. Setting $m=\lfloor\Real\beta'\rfloor+1$, we have by the formula \eqref{alternativePrabh}:
\begin{align*}
\prescript{C}{a}{\mathbb{D}}_{\alpha,\beta',\omega}^{\theta'}\circ\prescript{}{a}{\mathbb{I}}_{\alpha,\beta,\omega}^{\theta}f(t)&=\prescript{RL}{a}{\mathbb{D}}_{\alpha,\beta',\omega}^{\theta'}\left[\prescript{}{a}{\mathbb{I}}_{\alpha,\beta,\omega}^{\theta}f(t)-\sum_{j=0}^{m-1}\frac{t^j}{j!}\Big(\prescript{}{a}{\mathbb{I}}_{\alpha,\beta,\omega}^{\theta}f\Big)^{(j)}(a)\right] \\
&=\prescript{}{a}{\mathbb{I}}_{\alpha,\beta-\beta',\omega}^{\theta-\theta'}f(t)-\sum_{j=0}^{m-1}\Big(\prescript{}{a}{\mathbb{I}}_{\alpha,\beta,\omega}^{\theta}f\Big)^{(j)}(a)\cdot\prescript{RL}{a}{\mathbb{D}}_{\alpha,\beta',\omega}^{\theta'}\left(\frac{t^j}{j!}\right),
\end{align*}
where in the last line we used \eqref{thm2.5PrabFDE}. For each value of $j=0,1,\cdots,m-1$, since $j\leqslant m-1=\lfloor\Real\beta'\rfloor\leqslant\Real\beta'<\Real\beta$ and therefore $\Real (\beta-j)>0$, by \cite[Theorem 7]{generalizedfc} and the first statement of this Lemma, it follows that:
\[
\lim_{t\to a+}\Big(\prescript{}{a}{\mathbb{I}}_{\alpha,\beta,\omega}^{\theta}f\Big)^{(j)}(t)=\lim_{t\to a+}\left(\prescript{}{a}{\mathbb{I}}_{\alpha,\beta-j,\omega}^{\theta}f\right)(t)=0,
\]
which completes the proof.
\end{proof}

In the last part of Lemma \ref{importantproPrabFDE}, we have proved one composition relation for the Prabhakar operators, namely the Caputo-type derivative of the fractional integral. We will also need the converse, a formula for the fractional integral of the Caputo-type derivative, which will be stated in the following function space \cite{kilbas-marzan}:
\[
C^{\beta,m-1}[a,b]:=\left\{v\in C^{m-1}[a,b]\;:\; \prescript{C}{a}D^{\beta}v\text{ exists in }C[a,b]\right\}.
\]
Kilbas and Marzan used this space in \cite[\S3]{kilbas-marzan} for solving some Caputo fractional differential equations. It is a suitable setting because it guarantees the existence of Caputo fractional derivatives up to a given order without any further assumptions required. Given our context of Prabhakar operators, we shall endow it with the following norm:
\[
\|v\|_{C^{\beta,m-1}}=\sum_{k=0}^{m-1}\left\|v^{(k)}\right\|_{\infty}+\big\|\prescript{C}{a}{\mathbb{D}}_{\alpha,\beta,\omega}^{\theta}v\big\|_{\infty},
\]
where $\alpha,\beta,\theta,\omega\in\mathbb{C}$ such that $\Real (\alpha)>0$, $\Real (\beta)\geqslant0$, and $m-1\leqslant\Real \,\beta<m$. This function space is the same as the one used in \cite{analitical}, defined according to continuity of the classical Caputo derivative, but the norm is different, adapted for the Prabhakar setting. Note that the assumptions for this function space are enough to guarantee existence and continuity of the Caputo-type Prabhakar derivative:
\[
\prescript{C}{a}{\mathbb{D}}^{\theta}_{\alpha,\beta,\omega}v\in C[a,b]\quad\text{ for all }\;v\in C^{\beta}[a,b],
\]
because the series formula \eqref{PC:series} shows that $\prescript{C}{a}{\mathbb{D}}^{\theta}_{\alpha,\beta,\omega}v(t)$ is a uniformly convergent sum of the Caputo derivative $\prescript{C}{a}D^{\beta}v$ and various fractional integrals of it, which must all be continuous for $v\in C^{\beta}[a,b]$, since the fractional integral of a continuous function is continuous \cite{bonilla-trujillo-rivero}.

\begin{lem}\label{inversepPrabFDE}
If $\alpha,\beta,\theta,\omega\in\mathbb{C}$ with $\Real (\alpha)>0$ and $\Real (\beta)>0$ and $m=\lfloor\Real \beta\rfloor+1$, then for any $f\in C^{\beta,m-1}[a,b]$, we have
\[
\left(\prescript{}{a}{\mathbb{I}}_{\alpha,\beta,\omega}^{\theta}\circ\prescript{C}{a}{\mathbb{D}}_{\alpha,\beta,\omega}^{\theta}f\right)(t)=f(t)-\sum_{j=0}^{m-1}\frac{f^{(j)}(a)}{j!}\big(t-a\big)^j.
\]
In particular, if $0<\beta<1$ so that $m=1$, we have
\[
\left(\prescript{}{a}{\mathbb{I}}_{\alpha,\beta,\omega}^{\theta}\circ\prescript{C}{a}{\mathbb{D}}_{\alpha,\beta,\omega}^{\theta}f\right)(t)=f(t)-f(a).
\]
\end{lem}

\begin{proof}
This follows from the first definition \eqref{DPrab} of the Caputo-type derivative, together with the semigroup property \eqref{PI:semi} of Prabhakar integrals and the fundamental theorem of calculus.
\end{proof}

\subsection{Fractional calculus with respect to functions}

In order to make an extension of Prabhakar fractional calculus, let us now introduce the concept of fractional integrals and derivatives of a function with respect to another function.

In the classical Riemann--Liouville sense, the fractional integral of a function $f(t)$ with respect to a monotonically increasing $C^1$ function $\psi:[a,b]\to\mathbb{R}$ with $\psi'>0$ everywhere is defined \cite{osler} by
\[
\prescript{RL}{a}I^{\beta}_{\psi(t)}f(t)=\frac1{\Gamma(\beta)}\int_a^t \big(\psi(t)-\psi(s)\big)^{\beta-1}f(s)\psi'(s)\,\mathrm{d}s,\quad\Real (\beta)>0.
\]
This operator was first introduced by Osler \cite{osler}, and more detailed studies of both this fractional integral and its associated fractional derivatives can be found in \cite[\S2.5]{kilbas} and \cite[\S18.2]{samko}. One of its most important properties is its conjugation relation with the original Riemann--Liouville integral \eqref{fraci}:
\begin{equation}
\label{conjugation}
\prescript{RL}{a}I^{\beta}_{\psi(t)}=Q_\psi\circ\prescript{RL}{\psi(a)}I^{\beta}\circ Q_\psi^{-1},\quad\text{ where }Q_\psi:f\mapsto f\circ \psi.
\end{equation}
This enables many properties of the fractional integral with respect to $\psi$, such as composition relations, to be proved immediately from the corresponding properties of the Riemann--Liouville fractional integral. Conjugation relations like \eqref{conjugation} are also valid for the Riemann--Liouville and Caputo derivatives with respect to functions, and these relations can be used for efficient treatment of fractional differential equations with respect to functions \cite{fahad-rehman-fernandez,zaky-hendy-suragan}. The same idea of conjugation relations has also been applied to other types of fractional calculus \cite{agrawal,fahad-fernandez-rehman-siddiqi}, and more general fractional integral and derivative operators have also been taken with respect to functions \cite{oumarou-fahad-djida-fernandez}, illustrating the scope of this idea's applicability.

The Prabhakar fractional integral and derivatives of a function with respect to another function were first defined in \cite{fb:ssrn} and studied in more detail in \cite{oliveira1,oliveira2}:
\begin{align}
\prescript{}{a}{\mathbb{I}}_{\alpha,\beta,\omega}^{\theta;\psi(t)}f(t)&=\int_a^t \big(\psi(t)-\psi(s))^{\beta-1}E^{\theta}_{\alpha,\beta}\left(\omega\big(\psi(t)-\psi(s)\big)^{\alpha}\right)f(s)\psi'(s)\,\mathrm{d}s, \label{Pwrtf:int} \\
\prescript{RL}{a}{\mathbb{D}}_{\alpha,\beta,\omega}^{\theta;\psi(t)}f(t)&=\left(\frac{1}{\psi'(t)}\cdot\frac{\mathrm{d}}{\mathrm{d}t}\right)^m\left(\prescript{}{a}{\mathbb{I}}_{\alpha,m-\beta,\omega}^{-\theta;\psi(t)}f(t)\right), \label{Pwrtf:Rder} \\
\prescript{C}{a}{\mathbb{D}}_{\alpha,\beta,\omega}^{\theta;\psi(t)}f(t)&=\prescript{}{a}{\mathbb{I}}_{\alpha,m-\beta,\omega}^{-\theta;\psi(t)}\left(\left(\frac{1}{\psi'(t)}\cdot\frac{\mathrm{d}}{\mathrm{d}t}\right)^mf(t)\right), \label{Pwrtf:Cder}
\end{align}
where $\Real\alpha>0$ in every case, $\Real\beta>0$ in \eqref{Pwrtf:int}, and $\Real\beta\geqslant0$ with $m=\lfloor\Real\beta\rfloor+1$ in \eqref{Pwrtf:Rder}--\eqref{Pwrtf:Cder}.

Various properties of these operators were proved in \cite{oliveira1,oliveira2}, but those studies did not take account of the conjugation relation connecting these operators back to the original Prabhakar operators. We note that Prabhakar fractional calculus is a special case of fractional calculus with general analytic kernels \cite{fernandez-ozarslan-baleanu}, which has been extended to a version taken with respect to functions \cite{oumarou-fahad-djida-fernandez}, where a conjugation relation analogous to \eqref{conjugation} has been proved. Therefore, the corresponding relation holds for Prabhakar fractional integrals as a special case:
\begin{align*}
\prescript{}{a}{\mathbb{I}}_{\alpha,\beta,\omega}^{\theta;\psi(t)}&=Q_\psi\circ\prescript{}{\psi(a)}{\mathbb{I}}_{\alpha,\beta,\omega}^{\theta}\circ Q_\psi^{-1}, \\
\prescript{RL}{a}{\mathbb{D}}_{\alpha,\beta,\omega}^{\theta;\psi(t)}&=Q_\psi\circ\prescript{RL}{\psi(a)}{\mathbb{D}}_{\alpha,\beta,\omega}^{\theta}\circ Q_\psi^{-1}, \\
\prescript{C}{a}{\mathbb{D}}_{\alpha,\beta,\omega}^{\theta;\psi(t)}&=Q_\psi\circ\prescript{C}{\psi(a)}{\mathbb{D}}_{\alpha,\beta,\omega}^{\theta}\circ Q_\psi^{-1},\end{align*}
where the functional operator $Q_\psi$ is defined in \eqref{conjugation}. From these conjugation relations, all properties proved above for Prabhakar operators immediately give rise to corresponding properties for Prabhakar operators with respect to functions. For example, \eqref{alternativePrabh} implies that
\begin{equation*}
\prescript{C}{a}{\mathbb{D}}_{\alpha,\beta,\omega}^{\theta;\psi(t)}f(t)=\prescript{RL}{a}{\mathbb{D}}_{\alpha,\beta,\omega}^{\theta;\psi(t)}\left[f(t)-\sum_{j=0}^{m-1}\frac{\big(\psi(t)-\psi(a)\big)^j}{j!}\lim_{t\to a+}\left(\frac{1}{\psi'(t)}\cdot\frac{\mathrm{d}}{\mathrm{d}t}\right)^jf(t)\right],
\end{equation*}
with $\alpha,\beta,m$ as before. Or again, Lemma \ref{importantproPrabFDE} implies that
\[
\prescript{C}{a}{\mathbb{D}}_{\alpha,\beta',\omega}^{\theta',\psi(t)}\circ\prescript{}{a}{\mathbb{I}}_{\alpha,\beta,\omega}^{\theta,\psi(t)}f(t)=\prescript{}{a}{\mathbb{I}}_{\alpha,\beta-\beta',\omega}^{\theta-\theta',\psi(t)}f(t)
\]
where $\Real\alpha>0$ and $\Real\beta>\Real\beta'\geqslant0$ and $\theta,\theta'\in\mathbb{C}$, while Lemma \ref{inversepPrabFDE} implies that
\[
\left(\prescript{}{a}{\mathbb{I}}_{\alpha,\beta,\omega}^{\theta;\psi(t)}\circ\prescript{C}{a}{\mathbb{D}}_{\alpha,\beta,\omega}^{\theta;\psi(t)}f\right)(t)=f(t)-\sum_{j=0}^{m-1}\frac{\big(\psi(t)-\psi(a)\big)^j}{j!}\lim_{t\to a+}\left(\frac{1}{\psi'(t)}\cdot\frac{\mathrm{d}}{\mathrm{d}t}\right)^jf(t),
\]
with $\alpha,\beta,m$ as before and $f$ in the function space 
\[
C^{\beta,m-1}_{\psi(t)}[a,b]:=\left\{v\in C^{m-1}[a,b]\;:\; \prescript{C}{a}D^{\beta}_{\psi(t)}v(t)\text{ exists in }C[a,b]\right\},
\]
endowed with the norm
\[
\|v\|_{C^{\beta,m-1}_\psi}=\sum_{k=0}^{m-1}\left\|\left(\frac{1}{\psi'(t)}\cdot\frac{\mathrm{d}}{\mathrm{d}t}\right)^kv(t)\right\|_{\infty}+\big\|\prescript{C}{a}{\mathbb{D}}_{\alpha,\beta,\omega}^{\theta;\psi(t)}v(t)\big\|_{\infty}.
\]
It can be proved 
that the functional operator $Q_\psi$ is a natural isometry from the normed space $C^{\beta,m-1}[a,b]$ to the normed space $C^{\beta,m-1}_{\psi(t)}[a,b]$.

\section{Main results}\label{mainPrabFDE}

We will study the following differential equation with continuous variable coefficients and Caputo--Prabhakar fractional derivatives: 
\begin{equation}\label{eq1PrabFDE}
\prescript{C}{0}{\mathbb{D}}_{\alpha,\beta_0,\omega}^{\theta_0}v(t)+\sum_{i=1}^{m}\sigma_i(t)\prescript{C}{0}{\mathbb{D}}_{\alpha,\beta_i,\omega}^{\theta_i}v(t)=g(t),\quad t\in[0,T],
\end{equation}
to be solved for the unknown function $v(t)$, under the initial conditions 
\begin{equation}\label{eq2PrabFDE}
\frac{\mathrm{d}^k}{\mathrm{d}t^k} v(t)\Big|_{t=0+}=v^{(k)}(0)=e_k\in\mathbb{C},\quad k=0,1,\ldots,n_0-1, 
\end{equation}
where $\alpha,\beta_i,\theta_i,\omega\in\mathbb{C}$ with $\Real (\alpha)>0$ and $\Real (\beta_0)>\Real (\beta_1)>\cdots>\Real (\beta_{m})\geqslant0$ 

and $n_i=\lfloor \Real \beta_i\rfloor+1\in\mathbb{N}$ and the functions $\sigma_i,g\in C[0,T]$ for $i=0,1,\ldots,m$. 

We will also study the homogeneous case  
\begin{equation}\label{eq3PrabFDE}
\prescript{C}{0}{\mathbb{D}}_{\alpha,\beta_0,\omega}^{\theta_0}v(t)+\sum_{i=1}^{m}\sigma_i(t)\prescript{C}{0}{\mathbb{D}}_{\alpha,\beta_i,\omega}^{\theta_i}v(t)=0,\quad t\in[0,T],
\end{equation}
and the homogeneous initial conditions 
\begin{equation}\label{eq4PrabFDE}
v^{(k)}(0)=0,\quad k=0,1,\ldots,n_0-1,
\end{equation}
in order to obtain complementary functions which can then be used to construct the general solution.

\begin{defn}
A set of functions $v_j(t)$, $j=0,1,\ldots,n_0-1$, is called a canonical set of solutions of the homogeneous equation \eqref{eq3PrabFDE} if every function $v_j$ satisfies \eqref{eq3PrabFDE} and the following initial conditions hold for $j,k=0,1,\ldots,n_0-1$:
\begin{equation}
\label{initcond:canonical}
v_j^{(k)}(0)=
\begin{cases}
1,&\quad j=k,\\
0,&\quad j\neq k.
\end{cases}
\end{equation}
\end{defn}

We now study the existence, uniqueness, and representation of solutions for the above initial value problem.

\subsection{The general FDE with homogeneous initial conditions}
We start by proving the existence and uniqueness of solutions for the general FDE \eqref{eq1PrabFDE} with homogeneous initial conditions \eqref{eq4PrabFDE}.
 
\begin{thm}\label{lem3.1PrabFDE}
Let $\alpha,\beta_i,\theta_i,\omega\in\mathbb{C}$ with $\Real (\alpha)>0$ and $\Real (\beta_0)>\Real (\beta_1)>\cdots>\Real (\beta_{m})\geqslant0$ and $\Real (\beta_0)\not\in\mathbb{Z}$, and let $n_i=\lfloor \Real \beta_i\rfloor+1\in\mathbb{N}$ and the functions $\sigma_i,g\in C[0,T]$ for $i=0,1,\ldots,m$. Then the FDE \eqref{eq1PrabFDE} under the conditions \eqref{eq4PrabFDE} has a unique solution $v\in C^{\beta_0,n_0-1}[0,T]$, and it is represented by the following uniformly convergent series:
\begin{equation}\label{for27}
v(t)=\sum_{k=0}^{\infty}(-1)^k \prescript{}{0}{\mathbb{I}}_{\alpha,\beta_0,\omega}^{\theta_0}\left(\sum_{i=1}^{m}\sigma_i(t)\prescript{}{0}{\mathbb{I}}_{\alpha,\beta_0-\beta_i,\omega}^{\theta_0-\theta_i}\right)^{k}g(t).
\end{equation}
\end{thm}

\begin{proof}
Our proof will be in four parts: first transforming the FDE \eqref{eq1PrabFDE} with the conditions \eqref{eq4PrabFDE} into an equivalent integral equation, much easier to handle and work with; then using the Banach fixed point theorem to show that this integral equation has a unique solution in an appropriate function space; then constructing an appropriately convergent sequence of functions to give the unique solution function as a limit; and finally constructing an explicit formula for the solution function as an infinite series.

\medskip \textbf{Equivalent integral equation.} Let us take $v\in C^{\beta_0,n_0-1}[0,T]$ satisfying \eqref{eq1PrabFDE} and \eqref{eq4PrabFDE}. For $u(t)=\prescript{C}{0}{\mathbb{D}}_{\alpha,\beta_0,\omega}^{\theta_0}v(t)$, we know that $u\in C[0,T]$ by definition of the function space $C^{\beta_0,n_0-1}[0,T]$. By Lemma \ref{inversepPrabFDE} and conditions \eqref{eq4PrabFDE}, it follows that
\[
\prescript{}{0}{\mathbb{I}}_{\alpha,\beta_0,\omega}^{\theta_0}u(t)=\prescript{}{0}{\mathbb{I}}_{\alpha,\beta_0,\omega}^{\theta_0}\prescript{C}{0}{\mathbb{D}}_{\alpha,\beta_0,\omega}^{\theta_0}v(t)=v(t).\]
Due to $u\in C[0,T]$, $\Real (\beta_0)>\Real (\beta_{\it i})\geqslant0$, and Lemma \ref{importantproPrabFDE}, we have
\[\prescript{C}{0}{\mathbb{D}}_{\alpha,\beta_i,\omega}^{\theta_i}v(t)=\prescript{C}{0}{\mathbb{D}}_{\alpha,\beta_i,\omega}^{\theta_i}\prescript{}{0}{\mathbb{I}}_{\alpha,\beta_0,\omega}^{\theta_0}u(t)=\prescript{}{0}{\mathbb{I}}_{\alpha,\beta_0-\beta_i,\omega}^{\theta_0-\theta_i}u(t),\quad i=1,\ldots,m.\]
Therefore, equation \eqref{eq1PrabFDE} becomes
\begin{equation}\label{integraleqPrabFDE}
u(t)+\sum_{i=1}^{m}\sigma_i(t)\prescript{}{0}{\mathbb{I}}_{\alpha,\beta_0-\beta_i,\omega}^{\theta_0-\theta_i}u(t)=g(t).
\end{equation}
Thus, if $v\in C^{\beta_0,n_0-1}[0,T]$ is a solution of the initial value problem \eqref{eq1PrabFDE} and \eqref{eq4PrabFDE}, then $u=\prescript{C}{0}{\mathbb{D}}_{\alpha,\beta_0,\omega}^{\theta_0}v\in C[0,T]$ is a solution of the integral equation \eqref{integraleqPrabFDE}.

We now focus on the converse statement. Let $u\in C[0,T]$ be a solution of \eqref{integraleqPrabFDE}. By the application of the operator $\prescript{}{0}{\mathbb{I}}_{\alpha,\beta_0,\omega}^{\theta_0}$ to equation \eqref{integraleqPrabFDE}, we get
\begin{equation}
\label{equiv:step}
\prescript{}{0}{\mathbb{I}}_{\alpha,\beta_0,\omega}^{\theta_0}u(t)+\prescript{}{0}{\mathbb{I}}_{\alpha,\beta_0,\omega}^{\theta_0}\sum_{i=1}^{m}\sigma_i(t)\prescript{}{0}{\mathbb{I}}_{\alpha,\beta_0-\beta_i,\omega}^{\theta_0-\theta_i}u(t)=\prescript{}{0}{\mathbb{I}}_{\alpha,\beta_0,\omega}^{\theta_0}g(t).
\end{equation}
Defining $v(t)=\prescript{}{0}{\mathbb{I}}_{\alpha,\beta_0,\omega}^{\theta_0}u(t)$, from Lemma \ref{importantproPrabFDE} we obtain 
\[
\prescript{C}{0}{\mathbb{D}}_{\alpha,\beta_i,\omega}^{\theta_i}v(t)=\prescript{}{0}{\mathbb{I}}_{\alpha,\beta_0-\beta_i,\omega}^{\theta_0-\theta_i}u(t)\quad\text{and}\quad\prescript{}{0}{\mathbb{I}}_{\alpha,\beta_0-\beta_i,\omega}^{\theta_0-\theta_i}u\in C[0,T],
\]
therefore \eqref{equiv:step} implies
\[
v(t)+\prescript{}{0}{\mathbb{I}}_{\alpha,\beta_0,\omega}^{\theta_0}\sum_{i=1}^{m}\sigma_i(t)\prescript{C}{0}{\mathbb{D}}_{\alpha,\beta_i,\omega}^{\theta_i}v(t)=\prescript{}{0}{\mathbb{I}}_{\alpha,\beta_0,\omega}^{\theta_0}g(t).\]
Then, applying the Caputo--Prabhakar derivative:
\begin{equation*}
\prescript{C}{0}{\mathbb{D}}_{\alpha,\beta_0,\omega}^{\theta_0}v(t)+\prescript{C}{0}{\mathbb{D}}_{\alpha,\beta_0,\omega}^{\theta_0}\prescript{}{0}{\mathbb{I}}_{\alpha,\beta_0,\omega}^{\theta_0}\sum_{i=1}^{m}\sigma_i(t)\prescript{C}{0}{\mathbb{D}}_{\alpha,\beta_i,\omega}^{\theta_i}v(t)=\prescript{C}{0}{\mathbb{D}}_{\alpha,\beta_0,\omega}^{\theta_0}\prescript{}{0}{\mathbb{I}}_{\alpha,\beta_0,\omega}^{\theta_0}g(t).
\end{equation*}
By Lemma \ref{importantproPrabFDE}, we arrive at
\[
\prescript{C}{0}{\mathbb{D}}_{\alpha,\beta_0,\omega}^{\theta_0}v(t)+\sum_{i=1}^{m}\sigma_i(t)\prescript{C}{0}{\mathbb{D}}_{\alpha,\beta_i,\omega}^{\theta_i}v(t)=g(t),
\]
which is exactly \eqref{eq1PrabFDE}.

Moreover, by \cite[Theorem 7]{generalizedfc}, Lemma \ref{importantproPrabFDE}, and $\Real(\beta_0)\not\in\mathbb{Z}$ so that $\Real (\beta_0)>n_0-1$, we have
\[\frac{\mathrm{d}^k}{\mathrm{d}t^k} v(t)\Big|_{t=0+}=\frac{\mathrm{d}^k}{\mathrm{d}t^k} \prescript{}{0}{\mathbb{I}}_{\alpha,\beta_0,\omega}^{\theta_0}u(t)\Big|_{t=0+}=\prescript{}{0}{\mathbb{I}}_{\alpha,\beta_0-k,\omega}^{\theta_0}u(t)|_{t=0+}=0,\]
for any $k=0,1,\ldots,n_0-1$, giving the required initial conditions \eqref{eq4PrabFDE}, and we also have the required regularity (function space) since $\prescript{C}{0}{\mathbb{D}}_{\alpha,\beta_0,\omega}^{\theta_0}v=\prescript{C}{0}{\mathbb{D}}_{\alpha,\beta_0,\omega}^{\theta_0}\prescript{}{0}{\mathbb{I}}_{\alpha,\beta_0,\omega}^{\theta_0}u=u\in C[0,T]$ so that $v\in C^{\beta_0,n_0-1}[0,T]$. Thus, a solution $u\in C[0,T]$ of equation \eqref{integraleqPrabFDE} provides a solution $v=\prescript{}{0}{\mathbb{I}}_{\alpha,\beta_0,\omega}^{\theta_0}u\in C^{\beta_0,n_0-1}[0,T]$ for the equation \eqref{eq1PrabFDE} under the conditions \eqref{eq4PrabFDE}.

Finally, we have proved the equivalence of \eqref{eq1PrabFDE} and \eqref{eq4PrabFDE} with \eqref{integraleqPrabFDE}, under suitable regularity (function space) conditions on both sides of the equivalence.

\medskip \textbf{Existence and uniqueness.} Consider the operator $\mathfrak{T}$ defined by 
\[\mathfrak{T}u(t):=g(t)-\sum_{i=1}^{m}\sigma_i(t)\prescript{}{0}{\mathbb{I}}_{\alpha,\beta_0-\beta_i,\omega}^{\theta_0-\theta_i}u(t).\]
The integral equation \eqref{integraleqPrabFDE} is equivalent to $\mathfrak{T}u(t)=u(t)$, and it is clear that $\mathfrak{T}:C[0,T]\to C[0,T]$. Let us consider the norm on $C[0,T]$ defined by
\[
\|z\|_{p}:=\max_{t\in[0,T]}\Big(e^{-pt}|z(t)|\Big),
\]
for some large $p\in\mathbb{R}_+$ (to be fixed later according to our needs). This norm is equivalent to the supremum norm on $C[0,T]$, therefore $C[0,T]$ is a complete metric space under this norm. For the next estimates, we need to recall the following inequality:
\begin{equation}\label{util}
\Big|\prescript{RL}{0}I^{\lambda}e^{pt}\Big|\leqslant \frac{\Gamma(\Real\lambda)}{\left|\Gamma(\lambda)\right|}\cdot\frac{e^{pt}}{p^{\Real\lambda}}, \quad t,p\in \mathbb{R}_+,\;\Real\lambda>0,
\end{equation}
which follows from a simple inequality of integrals:
\[
\left|\Gamma(\lambda)\cdot\prescript{RL}{0}I^{\lambda}e^{pt}\right|\leqslant\Gamma(\Real\lambda)\cdot\prescript{RL}{-\infty}I^{\Real\lambda}e^{pt}=\Gamma(\Real\lambda)\cdot\frac{e^{pt}}{p^{\Real\lambda}}.
\]

Now, for any fixed $t\in [0,T]$ and $u_1,u_2\in C[0,T]$ and $p\in\mathbb{R}_+$, we get 
\begin{align*}
|\mathfrak{T}&u_1(t)-\mathfrak{T}u_2(t)| \\
&\leqslant\sum_{i=1}^{m}\|\sigma_i\|_{\infty}\sum_{k=0}^{\infty}\frac{|(\theta_0-\theta_i)_k||\omega|^k}{k!}\Big|\prescript{RL}{0}I^{\alpha k+\beta_0-\beta_i}\big[u_1(t)-u_2(t)\big]\Big| \\
&\leqslant\|u_1-u_2\|_{p}\sum_{i=1}^{m}\|\sigma_i\|_{\infty}\sum_{k=0}^{\infty}\frac{|(\theta_0-\theta_i)_k||\omega|^k}{k!}\Big|\prescript{RL}{0}I^{\alpha k+\beta_0-\beta_i}\big[e^{pt}\big]\Big| \\
&\leqslant\|u_1-u_2\|_{p}\sum_{i=1}^{m}\|\sigma_i\|_{\infty}\sum_{k=0}^{\infty}\frac{|(\theta_0-\theta_i)_k||\omega|^k}{k!}\cdot\frac{\Gamma(\Real(\beta_0-\beta_i+\alpha k))}{\left|\Gamma(\beta_0-\beta_i+\alpha k)\right|}\cdot\frac{e^{pt}}{p^{\Real(\beta_0-\beta_i)+\Real(\alpha)k}} \\
&=e^{pt}\|u_1-u_2\|_{p}\sum_{i=1}^{m}\frac{\|\sigma_i\|_{\infty}}{p^{\Real(\beta_0-\beta_i)}}\sum_{k=0}^{\infty}\frac{|(\theta_0-\theta_i)_k|}{k!}\cdot\frac{\Gamma(\Real(\beta_0-\beta_i)+k\Real\alpha))}{\left|\Gamma(\beta_0-\beta_i+\alpha k)\right|}\left(\frac{|\omega|}{p^{\Real\alpha}}\right)^k \\
&\leqslant Ce^{pt}\|u_1-u_2\|_{p},
\end{align*}
where $C>0$ is a constant, independent of $u_1,u_2$ and $t$, which can be taken to satisfy $0<C<1$ if we choose $p\in\mathbb{R}_+$ sufficiently large, since the $\beta_i$ and $\theta_i$ and $\sigma_i$ and $\alpha$ are fixed.

Thus, dividing by $e^{pt}$ in this inequality and taking the supremum over $t\in[0,T]$, we find
\[
\|\mathfrak{T}u_1-\mathfrak{T}u_2\|_{p}\leqslant C\|u_1-u_2\|_{p},
\]
which means that $T$ is contractive with respect to the norm $\|\cdot\|_{p}$. Equivalently, it is contractive with respect to the supremum norm $\|\cdot\|_{\infty}$ on $C[0,T]$. By applying the Banach fixed point theorem, it follows that the equation \eqref{integraleqPrabFDE} has a unique solution $u\in C[0,T]$ and the sequence $\{u_n(t)\}_{n\geqslant0}$ defined by
\begin{equation*}
\begin{cases}
u_0(t)&=g(t), \\
u_n(t)&=\displaystyle g(t)-\sum_{i=1}^{m}\sigma_i(t)\prescript{}{0}{\mathbb{I}}_{\alpha,\beta_0-\beta_i,\omega}^{\theta_0-\theta_i}u_{n-1}(t), \quad n=1,2,\ldots,
\end{cases}
\end{equation*}
converges (with respect to $\|\cdot\|_{\infty}$) to the limit $u$ in $C[0,T]$. Therefore, by the equivalence proved above, it follows that the initial value problem \eqref{eq1PrabFDE} and \eqref{eq4PrabFDE} has a unique solution $v\in C^{\beta_0,n_0-1}[0,T]$.

\medskip \textbf{Solution as a limit.} We already know that the sequence $\{u_n(t)\}_{n\geqslant0}$ converges in $C[0,T]$ with respect to $\|\cdot\|_{\infty}$. Since the Prabhakar fractional integral preserves uniform convergence, we have the following sequence also convergent with respect to $\|\cdot\|_{\infty}$:
\begin{equation*}
\begin{cases}
\prescript{}{0}{\mathbb{I}}_{\alpha,\beta_0,\omega}^{\theta_0}u_0(t)&=\displaystyle\prescript{}{0}{\mathbb{I}}_{\alpha,\beta_0,\omega}^{\theta_0}g(t), \\
\prescript{}{0}{\mathbb{I}}_{\alpha,\beta_0,\omega}^{\theta_0}u_n(t)&=\displaystyle\prescript{}{0}{\mathbb{I}}_{\alpha,\beta_0,\omega}^{\theta_0}g(t)-\prescript{}{0}{\mathbb{I}}_{\alpha,\beta_0,\omega}^{\theta_0}\sum_{i=1}^{m}\sigma_i(t)\prescript{}{0}{\mathbb{I}}_{\alpha,\beta_0-\beta_i,\omega}^{\theta_0-\theta_i}u_{n-1}(t). 
\end{cases}
\end{equation*}
Let us denote $v_{n}(t)=\prescript{}{0}{\mathbb{I}}_{\alpha,\beta_0,\omega}^{\theta_0}u_n(t)$ for all $n$. Therefore, by Lemma \ref{importantproPrabFDE} since $\Real\beta_0>\Real\beta_i\geqslant0$,
\[
\prescript{C}{0}{\mathbb{D}}_{\alpha,\beta_i,\omega}^{\theta_i}v_{n-1}(t)=\prescript{}{0}{\mathbb{I}}_{\alpha,\beta_0-\beta_i,\omega}^{\theta_0-\theta_i}u_{n-1}(t)\quad\text{ for all }\,n,
\]
and so we have the following sequence of functions $v_n$:
\begin{equation}\label{eq5eq6PrabFDE}
\begin{cases}
v_0(t)&=\displaystyle \prescript{}{0}{\mathbb{I}}_{\alpha,\beta_0,\omega}^{\theta_0}g(t), \\
v_n(t)&=\displaystyle v_0(t)-\prescript{}{0}{\mathbb{I}}_{\alpha,\beta_0,\omega}^{\theta_0}\sum_{i=1}^{m}\sigma_i(t)\prescript{C}{0}{\mathbb{D}}_{\alpha,\beta_i,\omega}^{\theta_i}v_{n-1}(t),\quad n=1,2,\ldots. 
\end{cases}
\end{equation}
Using Lemma \ref{importantproPrabFDE}, one can see that $v_n(t)\in C^{\beta_0,n_0-1}[0,T]$ for all $n$.

Now we prove the convergence of the sequence $\{v_n(t)\}_{n\geqslant0}$ in $C^{\beta_0,n_0-1}[0,T]$. Since $v_n(t)=\prescript{}{0}{\mathbb{I}}_{\alpha,\beta_0,\omega}^{\theta_0}u_n(t)$ and $\prescript{C}{0}{\mathbb{D}}_{\alpha,\beta_0,\omega}^{\theta_0}v_n(t)=u_n(t)$, and the same for $v$ and $u$, we get
\[\frac{\mathrm{d}^k}{\mathrm{d}t^k}\Big( v_n(t)-v(t)\Big)=\prescript{}{0}{\mathbb{I}}_{\alpha,\beta_0-k,\omega}^{\theta_0}\Big(u_n(t)-u(t)\Big),\quad k=0,1,\ldots,n_0-1,\]
where this is a fractional integral in each case because $\Real\beta_0\not\in\mathbb{Z}$ so $\Real(\beta_0-k)>0$ for all $k$. So we have
\[
\left\|\frac{\mathrm{d}^k}{\mathrm{d}t^k}\Big( v_n(t)-v(t)\Big)\right\|_{\infty}\leqslant \|u_n-u\|_{\infty}\int_0^T (t-s)^{\Real\beta_0-k-1}\big|E^{\theta_0}_{\alpha,\beta_0-k}(\omega(t-s)^{\alpha})\big|\,\mathrm{d}s,
\]
for $k=0,1,\ldots,n_0-1$, and of course $\left\|\prescript{C}{0}{\mathbb{D}}_{\alpha,\beta_0,\omega}^{\theta_0}(v_n-v)\right\|_{\infty}=\|u_n-u\|_{\infty}$. This gives 
\begin{align*}
\|v_n-v\|_{C^{\beta_0,n_0-1}}&=\sum_{k=0}^{n_0-1}\left\|\frac{\mathrm{d}^k}{\mathrm{d}t^k}(v_n-v)\right\|_{\infty}+\left\|\prescript{C}{0}{\mathbb{D}}_{\alpha,\beta_0,\omega}^{\theta_0}(v_n-v)\right\|_{\infty} \\
&\hspace{-1cm}\leqslant \left(1+\sum_{k=0}^{n_0-1}\int_0^T (t-s)^{\Real\beta_0-k-1}\big|E^{\theta_0}_{\alpha,\beta_0-k}(\omega(t-s)^{\alpha})\big|\,\mathrm{d}s\right)\|u_n-u\|_{\infty} \\
&\hspace{-1cm}\leqslant B\|u_n-u\|_{\infty},
\end{align*}
for some finite constant $B>0$. This implies that the sequence $\{v_n(t)\}_{n\geqslant0}$ converges in $C^{\beta_0,n_0-1}[0,T]$ with respect to $\|\cdot\|_{C^{\beta_0,n_0-1}}$, since we already know that the sequence $\{u_n(t)\}_{n\geqslant0}$ converges with respect to $\|\cdot\|_{\infty}$.

\medskip \textbf{Explicit solution function.} From \eqref{eq5eq6PrabFDE} and Lemma \ref{importantproPrabFDE}, the first approximation is given by 
\begin{align*}
v^1(t)&=\prescript{}{0}{\mathbb{I}}_{\alpha,\beta_0,\omega}^{\theta_0}g(t)-\prescript{}{0}{\mathbb{I}}_{\alpha,\beta_0,\omega}^{\theta_0}\sum_{i=1}^{m}\sigma_i(t)\prescript{C}{0}{\mathbb{D}}_{\alpha,\beta_i,\omega}^{\theta_i}\prescript{}{0}{\mathbb{I}}_{\alpha,\beta_0,\omega}^{\theta_0}g(t) \\
&=\prescript{}{0}{\mathbb{I}}_{\alpha,\beta_0,\omega}^{\theta_0}g(t)-\prescript{}{0}{\mathbb{I}}_{\alpha,\beta_0,\omega}^{\theta_0}\sum_{i=1}^{m}\sigma_i(t)\prescript{}{0}{\mathbb{I}}_{\alpha,\beta_0-\beta_i,\omega}^{\theta_0-\theta_i}g(t) \\
&=\sum_{k=0}^{1}(-1)^k \prescript{}{0}{\mathbb{I}}_{\alpha,\beta_0,\omega}^{\theta_0}\left(\sum_{i=1}^{m}\sigma_i(t)\prescript{}{0}{\mathbb{I}}_{\alpha,\beta_0-\beta_i,\omega}^{\theta_0-\theta_i}\right)^k g(t),
\end{align*}
where $v^1(t)\in C^{n_0-1,\beta_0,\theta_0}[0,T]$. Let us now suppose that for $n\in\mathbb{N}$ the $n$th approximation is given by
\begin{equation}
\label{nthapprox}
v^n(t)=\sum_{k=0}^{n}(-1)^k\prescript{}{0}{\mathbb{I}}_{\alpha,\beta_0,\omega}^{\theta_0}\left(\sum_{i=1}^{m}\sigma_i(t)\prescript{}{0}{\mathbb{I}}_{\alpha,\beta_0-\beta_i,\omega}^{\theta_0-\theta_i}\right)^k g(t),
\end{equation}
Then, using \eqref{eq5eq6PrabFDE}, the $(n+1)$th approximation is
\begin{align*}
v^{n+1}(t)&=\prescript{}{0}{\mathbb{I}}_{\alpha,\beta_0,\omega}^{\theta_0}g(t)-\prescript{}{0}{\mathbb{I}}_{\alpha,\beta_0,\omega}^{\theta_0}\sum_{i=1}^{m}\sigma_i(t)\prescript{C}{0}{\mathbb{D}}_{\alpha,\beta_i,\omega}^{\theta_i}v^{n}(t) \\
&=\prescript{}{0}{\mathbb{I}}_{\alpha,\beta_0,\omega}^{\theta_0}g(t)-\sum_{k=0}^{n}(-1)^k\prescript{}{0}{\mathbb{I}}_{\alpha,\beta_0,\omega}^{\theta_0}\sum_{i=1}^{m}\sigma_i(t) \\ &\hspace{3cm}\times\prescript{C}{0}{\mathbb{D}}_{\alpha,\beta_i,\omega}^{\theta_i}\prescript{}{0}{\mathbb{I}}_{\alpha,\beta_0,\omega}^{\theta_0}\left(\sum_{i=1}^{m}\sigma_i(t)\prescript{}{0}{\mathbb{I}}_{\alpha,\beta_0-\beta_i,\omega}^{\theta_0-\theta_i}\right)^k g(t) \\
&=\prescript{}{0}{\mathbb{I}}_{\alpha,\beta_0,\omega}^{\theta_0}g(t)+\sum_{k=0}^{n}(-1)^{k+1}\prescript{}{0}{\mathbb{I}}_{\alpha,\beta_0,\omega}^{\theta_0}\left(\sum_{i=1}^{m}\sigma_i(t)\prescript{}{0}{\mathbb{I}}_{\alpha,\beta_0-\beta_i,\omega}^{\theta_0-\theta_i}\right)^{k+1} g(t) \\
&=\sum_{k=0}^{n+1}(-1)^k\prescript{}{0}{\mathbb{I}}_{\alpha,\beta_0,\omega}^{\theta_0}\left(\sum_{i=1}^{m}\sigma_i(t)\prescript{}{0}{\mathbb{I}}_{\alpha,\beta_0-\beta_i,\omega}^{\theta_0-\theta_i}\right)^k g(t).
\end{align*}
This proves by induction that the formula \eqref{nthapprox} for $v_n$ is valid for all $n$. Therefore,
\[
v(t)=\lim_{n\to\infty}v^n (t)=\sum_{k=0}^{\infty}(-1)^k\prescript{}{0}{\mathbb{I}}_{\alpha,\beta_0,\omega}^{\theta_0}\left(\sum_{i=1}^{m}\sigma_i(t)\prescript{}{0}{\mathbb{I}}_{\alpha,\beta_0-\beta_i,\omega}^{\theta_0-\theta_i}\right)^k g(t),
\]
where the limit is taken in the norm $\|\cdot\|_{C^{\beta_0,n_0-1}}$ and therefore in particular the convergence is uniform.

\end{proof}

\subsection{Canonical set of solutions}

We now give the explicit representation for a canonical set of solutions of the homogeneous equation \eqref{eq3PrabFDE}. We will consider different cases of the fractional orders. A special collection of sets will help us to consider the possible cases:
\[\mathbb{W}_j:=\big\{i\in\{1,\dots,m\}\;:\;0\leqslant\Real (\beta_i)\leqslant j\big\},\quad j=0,1,\dots,n_0-1,\]
and we define $\varrho_j=\min(\mathbb{W}_j)$ for any $j$ such that $\mathbb{W}_j\neq\emptyset$. Thus, $\mathbb{W}_j\subseteq\mathbb{W}_{j+1}$ for all $j$, and we have $\varrho_j\leqslant i\Leftrightarrow\Real\beta_i\leqslant j$ for each $i,j$.
 
\begin{thm}\label{lem3.3PrabFDE}
Let $\alpha,\beta_i,\theta_i,\omega\in\mathbb{C}$ with $\Real (\alpha)>0$ and $\Real (\beta_0)>\Real (\beta_1)>\cdots>\Real (\beta_{m})\geqslant0$ and $\Real (\beta_0)\not\in\mathbb{Z}$, and let $n_i=\lfloor \Real \beta_i\rfloor+1\in\mathbb{N}$ and the functions $\sigma_i,g\in C[0,T]$ for $i=0,1,\ldots,m$. Then there exists a unique canonical set of solutions in $C^{\beta_0,n_0-1}[0,T]$ for the equation \eqref{eq3PrabFDE}, namely $v_j\in C^{\beta_0,n_0-1}[0,T]$ for $j=0,1,\ldots,n_0-1$ given by
\begin{equation}\label{form16}
v_j(t)=\frac{t^j}{j!}+\sum_{k=0}^{\infty} (-1)^{k+1}\prescript{}{0}{\mathbb{I}}_{\alpha,\beta_0,\omega}^{\theta_0}\left(\sum_{i=1}^{m}\sigma_i(t)\prescript{}{0}{\mathbb{I}}_{\alpha,\beta_0-\beta_i,\omega}^{\theta_0-\theta_i}\right)^{k}\Phi_j(t),
\end{equation}
where $\Phi_j$ denotes the function defined in general by
\begin{equation}
\label{form17}
\Phi_j(t)=\sum_{i=\varrho_j}^{m}\sigma_i(t)\,t^{j-\beta_i}E_{\alpha,j-\beta_i+1}^{-\theta_i}(\omega t^\alpha),
\end{equation}
and it is worth noting the following special cases.
\begin{enumerate}
\item For the cases $j>\Real\beta_1$, we have $\varrho_j=1$:
\begin{equation}\label{form17:norho}
\Phi_j(t)=\sum_{i=1}^{m}\sigma_i(t)\,t^{j-\beta_i}E_{\alpha,j-\beta_i+1}^{-\theta_i}(\omega t^\alpha)\quad\text{ for }j=n_1,n_1+1,\ldots,n_0-1.
\end{equation}

\item For the cases $j<\Real\beta_m$, we have $\mathbb{W}_j=\emptyset$ and an empty sum $\Phi_j(t)=0$:
\begin{align} \label{form17:zero}
v_j(t)=\frac{t^j}{j!},\quad&\text{ for }j=0,1,\ldots,j_0,\text{ where } \\ \nonumber j_0&\in\{0,1,\ldots,n_0-2\}\text{ satisfies }j_0<\Real(\beta_m)\leqslant j_0+1.
\end{align}

\item If $n_0=n_1$ and $\beta_{m}=0$, then \eqref{form17:norho} defines $\Phi_j$ for all $j=0,1,\ldots,n_0-1$.

\item If $\Real (\beta_i)\geqslant n_0-1$ for all $i=1,\ldots,m$, so that $n_0=n_1=\ldots=n_m$, then $\Phi_j(t)=0$ and \eqref{form17:zero} defines $v_j$ for all $j=0,1,\ldots,n_0-1$.
\end{enumerate}
\end{thm}

\begin{proof}
Following a proof similar to that of Theorem \ref{lem3.1PrabFDE}, we can show that finding the canonical set of solutions of \eqref{eq3PrabFDE}, i.e. solving \eqref{eq3PrabFDE} under the initial conditions \eqref{initcond:canonical}, is equivalent to the homogeneous version ($g(t)=0$) of the integral equation \eqref{integraleqPrabFDE}, under the correspondence $u_j(t)=\prescript{C}{0}{\mathbb{D}}_{\alpha,\beta_0,\omega}^{\theta_0}v_j(t)$ and $v_j(t)=\frac{t^j}{j!}+\prescript{}{0}{\mathbb{I}}_{\alpha,\beta_0,\omega}^{\theta_0}u_j(t)$, noting that $\frac{t^j}{j!}$ is always in $C^{\beta_0,n_0-1}[0,T]$ and the other regularity conditions are obtained as in the proof of Theorem \ref{lem3.1PrabFDE}.

Since we already solved \eqref{integraleqPrabFDE} in the proof of Theorem \ref{lem3.1PrabFDE}, we can now immediately obtain that the canonical set of solutions of \eqref{eq3PrabFDE} is given by the limit as $n\to\infty$ of the following sequence derived from \eqref{eq5eq6PrabFDE}, for each $j=0,1,\ldots,n_0-1$:
\begin{equation}\label{eq10eq11}
\begin{cases}
v^0_j (t)=\displaystyle\frac{t^j}{j!}, \\
v^n_j (t)=\displaystyle v^0_j(t)-\prescript{}{0}{\mathbb{I}}_{\alpha,\beta_0,\omega}^{\theta_0}\sum_{i=1}^{m}\sigma_i(t)\prescript{C}{0}{\mathbb{D}}_{\alpha,\beta_i,\omega}^{\theta_i}v^{n-1}_j(t),\quad n=1,2,\ldots, 
\end{cases}
\end{equation}
For $j,k\in\mathbb{N}_0$ we have 
\begin{equation*}
\frac{\mathrm{d}^k}{\mathrm{d}t^k}\left(\frac{t^j}{j!}\right)\bigg|_{t=0+}=
\begin{cases}
1,&\quad k=j, \\
0,&\quad k\neq j. 
\end{cases}
\end{equation*}
By \eqref{alternativePrabh}, we know that
\[
\prescript{C}{a}{\mathbb{D}}_{\alpha,\beta,\omega}^{\theta}\left(\frac{t^j}{j!}\right)=\prescript{RL}{a}{\mathbb{D}}_{\alpha,\beta,\omega}^{\theta}\left[\frac{t^j}{j!}-\sum_{{\color{red}k}=0}^{n_i-1}\frac{t^{k}}{{k}!}\cdot\frac{\mathrm{d}^k}{\mathrm{d}t^k}\left(\frac{t^j}{j!}\right)\bigg|_{t=0+}\right]
\]
Thus, for $j=0,1,\ldots,n_1-1$ (we choose this range of values since $n_1\geqslant n_i$ for all $i$), we get
\begin{equation}\label{formula18}
\prescript{C}{0}{\mathbb{D}}_{\alpha,\beta_i,\omega}^{\theta_i}t^j=\begin{cases}
\prescript{RL}{0}{\mathbb{D}}_{\alpha,\beta_i,\omega}^{\theta_i}t^j&\quad \text{ if }\varrho_j\leqslant i\leqslant m\quad (j\geqslant n_i), \\
0&\quad\text{ if }1\leqslant i< \varrho_j\quad (j\leqslant n_i-1). 
\end{cases}
\end{equation}
For $j=n_1,\ldots,n_0-1$ (noting that this range of values exists only if $n_0>n_1$), we have $k\leqslant n_i-1<j$ for all $i=1,\ldots,m$, and hence
\[
\prescript{C}{0}{\mathbb{D}}_{\alpha,\beta_i,\omega}^{\theta_i}t^j=\prescript{RL}{0}{\mathbb{D}}_{\alpha,\beta_i,\omega}^{\theta_i}t^j,\quad i=1,\ldots,m.
\]

Now, from \eqref{eq10eq11}, the first approximation of $v_j(t)$ is given by
\[
v^1_j(t)=\begin{cases}
\displaystyle\frac{t^j}{j!}-\prescript{}{0}{\mathbb{I}}_{\alpha,\beta_0,\omega}^{\theta_0}\sum_{i=\varrho_j}^{m}\sigma_i(t)\prescript{RL}{0}{\mathbb{D}}_{\alpha,\beta_i,\omega}^{\theta_i}\left(\frac{t^j}{j!}\right),\quad j=0,1,\ldots,n_1-1, \\
\displaystyle\frac{t^j}{j!}-\prescript{}{0}{\mathbb{I}}_{\alpha,\beta_0,\omega}^{\theta_0}\sum_{i=1}^{m}\sigma_i(t)\prescript{RL}{0}{\mathbb{D}}_{\alpha,\beta_i,\omega}^{\theta_i}\left(\frac{t^j}{j!}\right),\quad j=n_1,n_1+1,\ldots,n_0-1.
\end{cases}
\]
It is now clear that $v_j^1\in C^{\beta_0,n_0-1}[0,T]$ for any $j=0,1,\ldots,n_0-1$.

Let us now suppose that for $n\in\mathbb{N}$ the $n$th approximation is given by
\begin{align*}
v^n_j(t)=\frac{t^j}{j!}+\sum_{k=0}^{n-1}(-1)^{k+1}\prescript{}{0}{\mathbb{I}}_{\alpha,\beta_0,\omega}^{\theta_0}\left(\sum_{i=1}^{m}\sigma_i(t)\prescript{}{0}{\mathbb{I}}_{\alpha,\beta_0-\beta_i,\omega}^{\theta_0-\theta_i}\right)^{k}\sum_{i=1}^{m}\sigma_i(t)\prescript{RL}{0}{\mathbb{D}}_{\alpha,\beta_i,\omega}^{\theta_i}\left(\frac{t^j}{j!}\right)
\end{align*}
for $j=n_1,n_1+1,\ldots,n_0-1$ and
\begin{align*}
v^n_j(t)=\frac{t^j}{j!}+\sum_{k=0}^{n-1}(-1)^{k+1}\prescript{}{0}{\mathbb{I}}_{\alpha,\beta_0,\omega}^{\theta_0}\left(\sum_{i=1}^{m}\sigma_i(t)\prescript{}{0}{\mathbb{I}}_{\alpha,\beta_0-\beta_i,\omega}^{\theta_0-\theta_i}\right)^{k}\sum_{i=\varrho_j}^{m}\sigma_i(t)\prescript{RL}{0}{\mathbb{D}}_{\alpha,\beta_i,\omega}^{\theta_i}\left(\frac{t^j}{j!}\right)
\end{align*}
for $j=0,1,\ldots,n_1-1,$ and $v_j^n\in C^{\beta_0,n_0-1}[0,T]$ for $j=0,1,\ldots,n_0-1$. Inductively, we shall prove the analogous formula for the $(n+1)$th approximation. For $j=n_1,n_1+1,\ldots,n_0-1$, we obtain it by
\begin{align*}
v_j^{n+1}(t)&=\frac{t^j}{j!}-\prescript{}{0}{\mathbb{I}}_{\alpha,\beta_0,\omega}^{\theta_0}\sum_{i=1}^{m}\sigma_i(t)\prescript{C}{0}{\mathbb{D}}_{\alpha,\beta_i,\omega}^{\theta_i}v_j^{n}(t) \\
&=\frac{t^j}{j!}-\prescript{}{0}{\mathbb{I}}_{\alpha,\beta_0,\omega}^{\theta_0}\sum_{i=1}^{m}\sigma_i(t)\prescript{C}{0}{\mathbb{D}}_{\alpha,\beta_i,\omega}^{\theta_i}\left(\frac{t^j}{j!}\right) \\
&\hspace{1cm}+\prescript{}{0}{\mathbb{I}}_{\alpha,\beta_0,\omega}^{\theta_0}\sum_{i=1}^{m}\sigma_i(t)\prescript{C}{0}{\mathbb{D}}_{\alpha,\beta_i,\omega}^{\theta_i}\Bigg(\sum_{k=0}^{n-1}(-1)^{k+2} \\ &\hspace{3cm}\times\prescript{}{0}{\mathbb{I}}_{\alpha,\beta_0,\omega}^{\theta_0}\left(\sum_{i=1}^{m}\sigma_i(t)\prescript{}{0}{\mathbb{I}}_{\alpha,\beta_0-\beta_i,\omega}^{\theta_0-\theta_i}\right)^{k}\sum_{i=1}^{m}\sigma_i(t)\prescript{RL}{0}{\mathbb{D}}_{\alpha,\beta_i,\omega}^{\theta_i}\left(\frac{t^j}{j!}\right)\Bigg)
\end{align*}
Using Lemma \ref{importantproPrabFDE} and \eqref{formula18}, this becomes
\begin{align*}
v_j^{n+1}(t)&=\frac{t^j}{j!}-\prescript{}{0}{\mathbb{I}}_{\alpha,\beta_0,\omega}^{\theta_0}\sum_{i=1}^{m}\sigma_i(t)\prescript{RL}{0}{\mathbb{D}}_{\alpha,\beta_i,\omega}^{\theta_i}\left(\frac{t^j}{j!}\right) \\
&\hspace{1cm}+\sum_{k=0}^{n-1}(-1)^{k+2}\prescript{}{0}{\mathbb{I}}_{\alpha,\beta_0,\omega}^{\theta_0}\sum_{i=1}^{m}\sigma_i(t) \\ &\hspace{3cm}\times\prescript{}{0}{\mathbb{I}}_{\alpha,\beta_0-\beta_i,\omega}^{\theta_0-\theta_i}\left(\sum_{i=1}^{m}\sigma_i(t)\prescript{}{0}{\mathbb{I}}_{\alpha,\beta_0-\beta_i,\omega}^{\theta_0-\theta_i}\right)^{k}\sum_{i=1}^{m}\sigma_i(t)\prescript{RL}{0}{\mathbb{D}}_{\alpha,\beta_i,\omega}^{\theta_i}\left(\frac{t^j}{j!}\right) \\
&=\frac{t^j}{j!}-\prescript{}{0}{\mathbb{I}}_{\alpha,\beta_0,\omega}^{\theta_0}\sum_{i=1}^{m}\sigma_i(t)\prescript{RL}{0}{\mathbb{D}}_{\alpha,\beta_i,\omega}^{\theta_i}\left(\frac{t^j}{j!}\right) \\
&\hspace{1cm}+\sum_{k=0}^{n-1}(-1)^{k+2}\prescript{}{0}{\mathbb{I}}_{\alpha,\beta_0,\omega}^{\theta_0}\left(\sum_{i=1}^{m}\sigma_i(t)\prescript{}{0}{\mathbb{I}}_{\alpha,\beta_0-\beta_i,\omega}^{\theta_0-\theta_i}\right)^{k+1}\sum_{i=1}^{m}\sigma_i(t)\prescript{RL}{0}{\mathbb{D}}_{\alpha,\beta_i,\omega}^{\theta_i}\left(\frac{t^j}{j!}\right) \\
&=\frac{t^j}{j!}+\sum_{k=0}^{n} (-1)^{k+1}\prescript{}{0}{\mathbb{I}}_{\alpha,\beta_0,\omega}^{\theta_0}\left(\sum_{i=1}^{m}\sigma_i(t)\prescript{}{0}{\mathbb{I}}_{\alpha,\beta_0-\beta_i,\omega}^{\theta_0-\theta_i}\right)^{k}\sum_{i=1}^{m}\sigma_i(t)\prescript{RL}{0}{\mathbb{D}}_{\alpha,\beta_i,\omega}^{\theta_i}\left(\frac{t^j}{j!}\right). 
\end{align*}
In the same manner, for $j=0,1,\ldots,n_1-1$, one can obtain the second approximation as
\begin{align*}
v_j^{n+1}(t)=\frac{t^j}{j!}+\sum_{k=0}^{n}(-1)^{k+1}\prescript{}{0}{\mathbb{I}}_{\alpha,\beta_0,\omega}^{\theta_0}\left(\sum_{i=1}^{m}\sigma_i(t)\prescript{}{0}{\mathbb{I}}_{\alpha,\beta_0-\beta_i,\omega}^{\theta_0-\theta_i}\right)^{k}\sum_{i=\varrho_j}^{m}\sigma_i(t)\prescript{RL}{0}{\mathbb{D}}_{\alpha,\beta_i,\omega}^{\theta_i}\left(\frac{t^j}{j!}\right).
\end{align*}
In either case, $v_j^{n+1}\in C^{\beta_0,n_0-1}[0,T]$ for all $j=0,1,\ldots,n_0-1$, and the induction process is complete.

By the same argument used at the end of the proof of Theorem \ref{lem3.1PrabFDE}, we have for each $j$ that $v_j=\displaystyle{\lim_{n\to\infty}v_j^n}\in C^{\beta_0,n_0-1}[0,T]$. We have now achieved the general formula \eqref{form16} for the solution function $v_j$, with the general expression \eqref{form17} for $\Phi_j$ and the special case \eqref{form17:norho} when $j=n_1,n_1+1,\ldots,n_0-1$, after taking into account the following fact:
\[
\prescript{RL}{0}{\mathbb{D}}_{\alpha,\beta_i,\omega}^{\theta_i}\left(\frac{t^j}{j!}\right)=t^{j-\beta_i}E_{\alpha,j-\beta_i+1}^{-\theta_i}(\omega t^\alpha),
\]
which is easily proved using the series formula \eqref{PR:series} and standard facts on Riemann--Liouville differintegrals of power functions. Note that $j-\beta_i+1$ has positive real part for every $i,j$ in the sum, since $i\geqslant\varrho_j$ and therefore $j\geqslant\Real\beta_i>\Real(\beta_i-1)$.

Other special cases mentioned in the Theorem follow by analysing carefully the expression \eqref{formula18} and the definition of the $\varrho_j$. We leave the details to the interested reader.
\end{proof}

\subsection{Explicit form for solutions in the general case}

We now have explicit formulae, both for the canonical set of solutions given by the homogeneous FDE \eqref{eq3PrabFDE} with unit initial conditions \eqref{initcond:canonical} (as found in Theorem \ref{lem3.3PrabFDE}), and for the solution to the inhomogeneous FDE \eqref{eq1PrabFDE} with homogeneous initial conditions \eqref{eq4PrabFDE} (as found in Theorem \ref{lem3.1PrabFDE}). Combining these two results, we can obtain an explicit formula for the solution of the general initial value problem given by the inhomogeneous FDE \eqref{eq1PrabFDE} with the general initial conditions \eqref{eq2PrabFDE}.

\begin{thm}\label{secondthmFDEprab}
Let $\alpha,\beta_i,\theta_i,\omega\in\mathbb{C}$ with $\Real (\alpha)>0$ and $\Real (\beta_0)>\Real (\beta_1)>\cdots>\Real (\beta_{m})\geqslant0$ and $\Real (\beta_0)\not\in\mathbb{Z}$, and let $n_i=\lfloor \Real \beta_i\rfloor+1\in\mathbb{N}$ and the functions $\sigma_i,g\in C[0,T]$ for $i=0,1,\ldots,m$. Then the general initial value problem \eqref{eq1PrabFDE} and \eqref{eq2PrabFDE} has a unique solution $v\in C^{\beta_0,n_0-1}[0,T]$ and it is represented by
\[v(t)=\sum_{j=0}^{n_0-1}e_j v_j(t)+V_h(t),\]
where the functions $v_j$ are the canonical set of solutions found in Theorem \ref{lem3.3PrabFDE} and the function $V_h$ is
\[
V_h(t):=\sum_{k=0}^{\infty}(-1)^k\prescript{}{0}{\mathbb{I}}_{\alpha,\beta_0,\omega}^{\theta_0}\left(\sum_{i=1}^{m}\sigma_i(t)\prescript{}{0}{\mathbb{I}}_{\alpha,\beta_0-\beta_i,\omega}^{\theta_0-\theta_i}\right)^{k}g(t).
\]
\end{thm}

\begin{proof}
This follows from Theorem \ref{lem3.3PrabFDE}, Theorem \ref{lem3.1PrabFDE}, and the superposition principle, noting that $V_h$ is exactly the function \eqref{for27} found in Theorem \ref{lem3.1PrabFDE}.
\end{proof}

\begin{rem}\label{rem1FDEPrab}
Setting $\theta=0$ reduces the Caputo--Prabhakar derivative to the classical Caputo derivative, and it is straightforward to check that our results in this section reduce to those of \cite{analitical} when $\theta=0$.
\end{rem}

%
%
%

\subsection{Extension to operators with respect to functions}

The results proved above can be generalised by replacing the Prabhakar integrals and derivatives by the same operators taken with respect to a general monotonic $C^1$ function $\psi(t)$ satisfying $\psi(0)=0$ and $\psi'>0$ everywhere, instead of just with respect to $t$. In the setting of these generalised operators, we write the FDE as follows:
\begin{equation}\label{WRTF:eq1PrabFDE}
\prescript{C}{0}{\mathbb{D}}_{\alpha,\beta_0,\omega}^{\theta_0;\psi(t)}v(t)+\sum_{i=1}^{m}\sigma_i(t)\prescript{C}{0}{\mathbb{D}}_{\alpha,\beta_i,\omega}^{\theta_i;\psi(t)}v(t)=g(t),\quad t\in[0,T],
\end{equation}
to be solved for the unknown function $v(t)$, under the general initial conditions 
\begin{equation}\label{WRTF:eq2PrabFDE}
\left(\frac{1}{\psi'(t)}\cdot\frac{\mathrm{d}}{\mathrm{d}t}\right)^k v(t)\bigg|_{t=0+}=e_k\in{\color{red}\mathbb{C}},\quad k=0,1,\ldots,n_0-1, 
\end{equation}
where $\alpha,\beta_i,\theta_i,\omega\in\mathbb{C}$ with $\Real (\alpha)>0$ and $\Real (\beta_0)>\Real (\beta_1)>\cdots>\Real (\beta_{m})\geqslant0$ and $n_i=\lfloor \Real \beta_i\rfloor+1\in\mathbb{N}$ and the functions $\sigma_i,g\in C[0,T]$ for $i=0,1,\ldots,m$.
We also consider the corresponding homogeneous FDE:
\begin{equation}\label{WRTF:eq3PrabFDE}
\prescript{C}{0}{\mathbb{D}}_{\alpha,\beta_0,\omega}^{\theta_0;\psi(t)}v(t)+\sum_{i=1}^{m}\sigma_i(t)\prescript{C}{0}{\mathbb{D}}_{\alpha,\beta_i,\omega;\psi(t)}^{\theta_i}v(t)=0,\quad t\in[0,T],
\end{equation}
and the homogeneous initial conditions 
\begin{equation}\label{WRTF:eq4PrabFDE}
\left(\frac{1}{\psi'(t)}\cdot\frac{\mathrm{d}}{\mathrm{d}t}\right)^k v(t)\bigg|_{t=0+}=0,\quad k=0,1,\ldots,n_0-1,
\end{equation}
The analogue of Theorem \ref{lem3.1PrabFDE} for this type of problem, with respect to a function $\psi$, is as follows.

\begin{thm}\label{WRTF:lem3.1PrabFDE}
Let $\psi\in C^1[0,\infty)$ be a monotonic function with $\psi(0)=0$ and $\psi'>0$ everywhere. Let $\alpha,\beta_i,\theta_i,\omega\in\mathbb{C}$ with $\Real (\alpha)>0$ and $\Real (\beta_0)>\Real (\beta_1)>\cdots>\Real (\beta_{m})\geqslant0$ and $\Real (\beta_0)\not\in\mathbb{Z}$, and let $n_i=\lfloor \Real \beta_i\rfloor+1\in\mathbb{N}$ and the functions $\sigma_i,g\in C[0,T]$ for $i=0,1,\ldots,m$. Then the FDE \eqref{eq1PrabFDE} under the conditions \eqref{eq4PrabFDE} has a unique solution $v\in C^{\beta_0,n_0-1}_{\psi}[0,T]$, and it is represented by the following uniformly convergent series:
\[
v(t)=\sum_{k=0}^{\infty}(-1)^k \prescript{}{0}{\mathbb{I}}_{\alpha,\beta_0,\omega}^{\theta_0;\psi(t)}\left(\sum_{i=1}^{m}\sigma_i(t)\prescript{}{0}{\mathbb{I}}_{\alpha,\beta_0-\beta_i,\omega}^{\theta_0-\theta_i;\psi(t)}\right)^{k}g(t).
\]
\end{thm}

\begin{proof}
The proof is identical to that of Theorem \ref{lem3.1PrabFDE}, except with all integrals, derivatives, operators, function spaces, etc. taken with respect to the function $\psi$. The first part of the proof (equivalence of the integral equation) is similar enough to be omitted entirely. For the second part, we note that the norm $\|\cdot\|_p$ would here be defined by
\[
\|z\|_{p}:=\max_{t\in[0,T]}\Big(e^{-p\psi(t)}|z(t)|\Big),
\]
and we would use the following estimate for the Riemann--Liouville integral with respect to $\psi$:
\[
\Big|\prescript{RL}{0}I^{\lambda}_{\psi(t)}e^{p\psi(t)}\Big|\leqslant \frac{\Gamma(\Real\lambda)}{\left|\Gamma(\lambda)\right|}\cdot\frac{e^{p\psi(t)}}{p^{\Real\lambda}}, \quad t,p\in \mathbb{R}_+,\;\Real\lambda>0,
\]
which follows immediately from \eqref{util} using the conjugation relations for operators with respect to $\psi$. This enables boundedness of the relevant linear operator to be shown in the same way as in the proof of Theorem \ref{lem3.1PrabFDE}. Finally, for the bounding of the integral in the third part of the proof, we will need to bound the multiplier $\psi'(t)$ as well as everything else, but this is perfectly possible since $\psi$ is assumed to be a $C^1$ function.
\end{proof}

Next, the analogue of Theorem \ref{lem3.3PrabFDE} for a canonical set of solutions with respect to a function $\psi$ is as follows.

\begin{thm}\label{WRTF:lem3.3PrabFDE}
Let $\psi\in C^1[0,\infty)$ be a monotonic function with $\psi(0)=0$  and $\psi'>0$ everywhere. Let $\alpha,\beta_i,\theta_i,\omega\in\mathbb{C}$ with $\Real (\alpha)>0$ and $\Real (\beta_0)>\Real (\beta_1)>\cdots>\Real (\beta_{m})\geqslant0$ and $\Real (\beta_0)\not\in\mathbb{Z}$, and let $n_i=\lfloor \Real \beta_i\rfloor+1\in\mathbb{N}$ and the functions $\sigma_i,g\in C[0,T]$ for $i=0,1,\ldots,m$. Then there exists a unique canonical set of solutions of equation \eqref{WRTF:eq3PrabFDE}, namely $v_{j,\psi}\in C^{\beta_0,n_0-1}_{\psi}[0,T]$ for $j=0,1,\ldots,n_0-1$ given by
\begin{equation}\label{WRTF:form16}
v_{j,\psi}(t)=\frac{\psi(t)^j}{j!}+\sum_{k=0}^{\infty} (-1)^{k+1}\prescript{}{0}{\mathbb{I}}_{\alpha,\beta_0,\omega}^{\theta_0;\psi(t)}\left(\sum_{i=1}^{m}\sigma_i(t)\prescript{}{0}{\mathbb{I}}_{\alpha,\beta_0-\beta_i,\omega}^{\theta_0-\theta_i;\psi(t)}\right)^{k}\Phi_{j}\big(\psi(t)\big),
\end{equation}
where $\Phi_j$ is the same function defined by \eqref{form17} in Theorem \ref{lem3.3PrabFDE}.
\end{thm}

\begin{proof}
Once again, the proof is identical to that of Theorem \ref{lem3.3PrabFDE} except {\color{red}that} everything {\color{red}is} taken with respect to the function $\psi$. We note in particular that
\begin{equation*}
\left(\frac{1}{\psi'(t)}\cdot\frac{\mathrm{d}}{\mathrm{d}t}\right)^k\left(\frac{\psi(t)^j}{j!}\right)\bigg|_{t=0+}=
\begin{cases}
1,&\quad k=j, \\
0,&\quad k\neq j,
\end{cases}
\end{equation*}
an invaluable result in constructing the canonical set of solution functions.
\end{proof}

Finally, the analogue of Theorem \ref{secondthmFDEprab} for the solution to a general initial value problem with respect to a function $\psi$ is as follows.

\begin{thm}\label{WRTF:secondthmFDEprab}
Let $\psi\in C^1[0,\infty)$ be a monotonic function with $\psi(0)=0$  and $\psi'>0$ everywhere. Let $\alpha,\beta_i,\theta_i,\omega\in\mathbb{C}$ with $\Real (\alpha)>0$ and $\Real (\beta_0)>\Real (\beta_1)>\cdots>\Real (\beta_{m})\geqslant0$ and $\Real (\beta_0)\not\in\mathbb{Z}$, and let $n_i=\lfloor \Real \beta_i\rfloor+1\in\mathbb{N}$ and the functions $\sigma_i,g\in C[0,T]$ for $i=0,1,\ldots,m$. Then the general initial value problem \eqref{WRTF:eq1PrabFDE} and \eqref{WRTF:eq2PrabFDE} has a unique solution $v\in C^{\beta_0,n_0-1}_{\psi}[0,T]$ and it is represented by
\[v(t)=\sum_{j=0}^{n_0-1}e_j v_{j,\psi}(t)+V_{h,\psi}(t),\]
where the functions $v_{j,\psi}$ are the canonical set of solutions found in Theorem \ref{WRTF:lem3.3PrabFDE} and the function $V_{h,\psi}$ is
\[
V_{h,\psi}(t):=\sum_{k=0}^{\infty}(-1)^k\prescript{}{0}{\mathbb{I}}_{\alpha,\beta_0,\omega}^{\theta_0;\psi(t)}\left(\sum_{i=1}^{m}\sigma_i(t)\prescript{}{0}{\mathbb{I}}_{\alpha,\beta_0-\beta_i,\omega}^{\theta_0-\theta_i;\psi(t)}\right)^{k}g(t).
\]
\end{thm}

\begin{proof}
Here the proof follows exactly the same lines as the previous proof in the case $\psi(t)=t$ (not with respect to a function). We omit the straightforward details.
\end{proof}

\section{Examples}\label{FDEPrabconstcoe}

In this section, to illustrate the general results achieved above, we will study the same initial value problems under the assumption that the coefficient functions $\sigma_i(t)$ are actually constant functions. Thus, we consider the following Prabhakar-type linear differential equation with constant coefficients:
\begin{equation}\label{eq1PrabFDEconst}
\prescript{C}{0}{\mathbb{D}}_{\alpha,\beta_0,\omega}^{\theta}v(t)+\sum_{i=1}^{m}\sigma_i\prescript{C}{0}{\mathbb{D}}_{\alpha,\beta_i,\omega}^{\theta}v(t)=g(t),\quad t\in[0,T],
\end{equation}
under the same initial conditions \eqref{eq2PrabFDE}, where $\sigma_i,\alpha,\beta_i,\theta,\omega\in\mathbb{C}$ with $\Real (\alpha)>0$ and $\Real (\beta_0)>\Real (\beta_1)>\cdots>\Real (\beta_{\it m})\geqslant0$ and $g\in C[0,T]$ and $n_i=\lfloor \Real \beta_i\rfloor+1\in\mathbb{N}$ for $i=0,1,\ldots,m$. 

Notice that we are not considering a whole set of different parameters $\theta_i$ in the equation \eqref{eq1PrabFDEconst}, as we did in the previous section. Instead, we have fixed $\theta_i=\theta$ for all $i=0,1,\ldots,m$. This is to enable us to get easier representations of the solution functions, using multivariate Mittag-Leffler functions that already exist in the literature \cite{luchko}, which we recall as follows.


\begin{defn}
The multivariate Mittag-Leffler function of $n$ complex variables $z_1,\ldots,z_n$, with $n+1$ complex parameters $\alpha_1,\ldots,\alpha_n,\beta$ satisfying $\Real\alpha_i,\Real\beta>0$ for $i=1,\ldots,n$, is defined by
\begin{align*}
E_{(\alpha_1,\ldots,\alpha_n),\beta}(z_1,\ldots,z_n)&=\sum_{k=0}^{\infty}\sum_{\substack{k_1+\cdots+k_n= k, \\ k_1,\ldots,k_n\geq0}}\frac{k!}{k_1!\times\cdots\times k_n!}\cdot\frac{\displaystyle{\prod_{i=1}^n z_i^{k_i}}}{\Gamma\left(\beta+\displaystyle{\sum_{i-1}^n\alpha_i k_i}\right)}, \\
&=\sum_{k_1,\ldots,k_n\geq0}\frac{(k_1+\cdots+k_n)!}{k_1!\times\cdots\times k_n!}\cdot\frac{z_1^{k_1}\times\cdots\times z_n^{k_n}}{\Gamma\left(\beta+\alpha_1 k_1+\cdots+\alpha_n k_n\right)},
\end{align*}
where the multiple series is locally uniformly convergent for all $(z_1,\cdots,z_n)\in\mathbb{C}^n$ under the given conditions on the parameters.
\end{defn}

We now establish the main results of this section.

\begin{thm}\label{thm3.1FDEPrabconst}
Let $\sigma_i,\alpha,\beta_i,\theta,\omega\in\mathbb{C}$ with $\Real (\alpha)>0$ and $\Real (\beta_0)>\Real (\beta_1)>\cdots>\Real (\beta_{\it m})\geqslant0$ and $\Real(\beta_0)\not\in\mathbb{Z}$, and let $n_i=\lfloor \Real \beta_i\rfloor+1\in\mathbb{N}$ for $i=0,1,\ldots,m$, and $g\in C[0,T]$. Then the FDE \eqref{eq1PrabFDEconst} with homogeneous initial conditions \eqref{eq4PrabFDE} has a unique solution $v\in C^{\beta_0,n_0-1}[0,T]$ and it is represented by
\begin{align*}
v(t)=\sum_{n=0}^{\infty}&\frac{(\theta)_n\omega^n}{n!}\int_0^t s^{\beta_0+\alpha n-1}\times \\ 
&\times E_{(\beta_0-\beta_1,\ldots,\beta_0-\beta_m),\alpha n+\beta_0}(-\sigma_1 s^{\beta_0-\beta_1},\ldots,-\sigma_m s^{\beta_0-\beta_m})g(t-s)ds.
\end{align*}
\end{thm}

\begin{proof}
By Theorem \ref{lem3.1PrabFDE}, we know that the FDE \eqref{eq1PrabFDEconst} with the conditions \eqref{eq4PrabFDE} has a unique solution $v\in C^{\beta_0,n_0-1}[0,T]$ given by
\[
v(t)=\sum_{k=0}^{\infty}(-1)^k \prescript{}{0}{\mathbb{I}}_{\alpha,\beta_0,\omega}^{\theta}\left(\sum_{i=1}^{m}\sigma_i\prescript{}{0}{\mathbb{I}}_{\alpha,\beta_0-\beta_i,\omega}^{0}\right)^{k}g(t).
\]
By the semigroup property of the Prabhakar fractional integral \eqref{PI:semi}, the process of taking a finite sum of Prabhakar integral terms and raising it to a finite power will look exactly like doing the same thing with just numbers, as the ``powers'' of Prabhakar integrals combine according to the semigroup property. (This can be formalised using Mikusi\'nski's operational calculus for an algebraic interpretation of Prabhakar operators \cite{rani-fernandez1}, but in this case it is clear from direct calculation.) So, we get a multinomial expansion leading to a multivariate Mittag-Leffler function as follows:
\begin{align}
v(t)&=\sum_{k=0}^{\infty}(-1)^k \prescript{}{0}{\mathbb{I}}_{\alpha,\beta_0,\omega}^{\theta}\left(\sum_{i=1}^{m}\sigma_i\prescript{}{0}{\mathbb{I}}_{\alpha,\beta_0-\beta_i,\omega}^{0}\right)^{k}g(t) \nonumber \\
&=\sum_{k=0}^{\infty}(-1)^k \prescript{}{0}{\mathbb{I}}_{\alpha,\beta_0,\omega}^{\theta}\left(\sum_{k_1+\cdots+k_m=k}\frac{k!}{k_1!\times\cdots\times k_m!}\prod_{i=1}^{m}\sigma_i^{k_i}\prescript{}{0}{\mathbb{I}}_{\alpha,(\beta_0-\beta_i)k_i,\omega}^{0}\right)g(t) \nonumber \\
&=\sum_{k=0}^{\infty}(-1)^k \left(\sum_{k_1+\cdots+k_m=k}\frac{k!}{k_1!\times\cdots\times k_m!}\prod_{i=1}^{m}\sigma_i^{k_i}\right)\prescript{}{0}{\mathbb{I}}_{\alpha,\beta_0+\sum_{i=1}^{m}(\beta_0-\beta_i)k_{i},\omega}^{\theta}g(t) \nonumber \\
&=\sum_{k=0}^{\infty}\left(\sum_{k_1+\cdots+k_m=k}\frac{k!\prod_{i=1}^{m}(-\sigma_i)^{k_i}}{k_1!\times\cdots\times k_m!}\right)\sum_{n=0}^{\infty}\frac{(\theta)_n\omega^n}{n!}\prescript{RL}{0}I^{n\alpha+\beta_0+\sum_{i=1}^{m}(\beta_0-\beta_i)k_{i}}g(t). \label{multisum}
\end{align}
Writing the Riemann--Liouville integral explicitly using the integral \eqref{fraci}, and changing the places of integrals and sums, we can arrive at:
\begin{multline*}
v(t)=\sum_{n=0}^{\infty}\frac{(\theta)_n\omega^n}{n!}\int_0^t s^{\beta_0+\alpha n-1} \times \\
\times E_{(\beta_0-\beta_1,\ldots,\beta_0-\beta_m),\alpha n+\beta_0}\left(-\sigma_1 s^{\beta_0-\beta_1},\ldots,-\sigma_m s^{\beta_0-\beta_m}\right)g(t-s)\,\mathrm{d}s,
\end{multline*}
which is the final answer.
\end{proof}

\begin{thm}
Let $\sigma_i,\alpha,\beta_i,\theta,\omega\in\mathbb{C}$ with $\Real (\alpha)>0$ and $\Real (\beta_0)>\Real (\beta_1)>\cdots>\Real (\beta_{\it m})\geqslant0$ and $\Real(\beta_0)\not\in\mathbb{Z}$, and let $n_i=\lfloor \Real \beta_i\rfloor+1\in\mathbb{N}$ for $i=0,1,\ldots,m$, and $g\in C[0,T]$. Then the FDE \eqref{eq1PrabFDEconst} with general initial conditions \eqref{eq2PrabFDE} has a unique solution $v\in C^{\beta_0,n_0-1}[0,T]$ and it is represented by:
\[v(t)=\sum_{j=0}^{n_0-1}e_j v_j(t)+V_h^{c}(t),\]
where $V_h^c$ is given by
\begin{multline*}
V_h^{c}(t):=\sum_{n=0}^{\infty}\frac{(\theta)_n\omega^n}{n!}\int_0^t s^{\beta_0+\alpha n-1} \times \\
\times E_{(\beta_0-\beta_1,\ldots,\beta_0-\beta_m),\alpha n+\beta_0}\left(-\sigma_1 s^{\beta_0-\beta_1},\ldots,-\sigma_m s^{\beta_0-\beta_m}\right)g(t-s)\,\mathrm{d}s,
\end{multline*}
and the canonical set of solutions $v_j$ to the constant-coefficient problem are defined by
\begin{multline*}
v_j(t)=\frac{t^j}{j!}-\sum_{n=0}^{\infty}\frac{(\theta)_n\omega^n}{n!}\int_0^t s^{\beta_0+\alpha n-1} \times \\
\times E_{(\beta_0-\beta_1,\ldots,\beta_0-\beta_m),\alpha n+\beta_0}\left(-\sigma_1 s^{\beta_0-\beta_1},\ldots,-\sigma_m s^{\beta_0-\beta_m}\right)\Phi_j(t-s)\,\mathrm{d}s,
\end{multline*}
with the function $\Phi_j$ defined in the same way as in Theorem \ref{lem3.3PrabFDE}, by the formula \eqref{form17} with special cases given by \eqref{form17:norho}, \eqref{form17:zero}, etc.
\end{thm}

\begin{proof}
This follows from the general result of Theorem \ref{secondthmFDEprab}. The function $V_h$ found in Theorem \ref{secondthmFDEprab} now becomes the function $V_h^c$ which we already simplified in Theorem \ref{thm3.1FDEPrabconst} above. By comparing the formulae \eqref{form16} and \eqref{for27}, we can observe that, in general, the function $v_j(t)$ is exactly $\frac{t^j}{j!}$ minus the function $V_h$ with $g$ replaced by $\Phi_j$. This gives us the expressions stated here for the $v_j$ functions.
\end{proof}

\begin{rem}
A general linear constant-coefficient Caputo--Prabhakar fractional differential equation of the form \eqref{eq1PrabFDEconst} was already studied and solved in \cite{rani-fernandez2}, using the method of Mikusi\'nski's operational calculus. The solution found there is consistent with the one we have found here, just expressed in a different form due to a different choice in how to manage the multiple sums.

What emerges in \eqref{multisum} is a multiple sum that combines a sum over $n$ (corresponding to the Prabhakar function) with a sum over $k_1,\cdots,k_m$ (corresponding to a multivariate Mittag-Leffler function). In our results above, we have simplified this to a single sum over $n$ of an expression involving multivariate Mittag-Leffler functions whose parameters depend on $n$. In the previous work of \cite{rani-fernandez2}, the same expression was simplified to a sum over $k_1,\cdots,k_m$ of an expression involving a Prabhakar function whose parameters depend on $k_1,\cdots,k_m$. These are two different valid choices for how to simplify the complicated expression, and so our work here complements that of \cite{rani-fernandez2} by expressing the same solution function in a different form, a single sum of multivariate Mittag-Leffler functions rather than a multiple sum of univariate Mittag-Leffler functions.
\end{rem}

\section{Acknowledgements}

\noindent The second and third authors were supported by the Nazarbayev University Program 091019CRP2120. The second author was also supported  by the FWO Odysseus 1 grant G.0H94.18N: Analysis and Partial Differential Equations and the Methusalem programme of the Ghent University Special Research Fund (BOF) (Grant number 01M01021).


\begin{thebibliography}{00}

\bibitem{agrawal} O.P. Agrawal, Some generalized fractional calculus operators and their applications in integral equations, Fract. Calc. Appl. Anal. 15(4) (2012), 700--711.

\bibitem{AML} C.N. Angstmann, B.I. Henry, Generalized fractional power series solutions for fractional differential equations, Appl. Math. Lett. 102 (2020).

\bibitem{vcserbia1} T.M. Atanackovi\'c, B. Stankovi\'c. Linear fractional differential equation with variable coefficients I. Bull. de l Acad. Serbe Sci. Arts, Cl. Math. 38 (2013), 27--42.

\bibitem{vcserbia2} T.M. Atanackovi\'c, B. Stankovi\'c. Linear fractional differential equation with variable coefficients II. Bull. de l Acad. Serbe Sci. Arts, Cl. Math. 39 (2014), 53--78.

\bibitem{BRS} D. Baleanu, J.E. Restrepo, D. Suragan. A class of time-fractional Dirac type operators. Chaos Solitons Fractals, 143, \#510590, (2021).

\bibitem{bonilla-trujillo-rivero}
B. Bonilla, J.J. Trujillo, M. Rivero. Fractional Order Continuity and Some Properties about Integrability and Differentiability of Real Functions. J. Math. Anal. Appl. 231 (1999), 205--212.

\bibitem{first} M.M. Dzhrbashyan, A.B. Nersessyan. Fractional derivatives and Cauchy problem for differential equations of fractional order, Izv. AN Arm. SSR. Mat. 3 (1968).

\bibitem{pocha} A. Erd\'elyi, W.F. Oberhettinger, F.G. Tricomi. Higher Transcendental Functions, Vol. I. McGraw-Hill, New York, 1953.



\bibitem{fahad-fernandez-rehman-siddiqi} H.M. Fahad, A. Fernandez, M. u. Rehman, M. Siddiqi, Tempered and Hadamard-type fractional calculus with respect to functions, Medit. J. Math. 18 (2021), 143.

\bibitem{fahad-rehman-fernandez}
H.M. Fahad, M. ur Rehman, A. Fernandez, On Laplace transforms with respect to functions and their applications to fractional differential equations, Math. Meth. Appl. Sci. (2021), 1--20.

\bibitem{fb:ssrn} A. Fernandez, D. Baleanu. Differintegration with respect to functions in fractional models involving Mittag-Leffler functions. SSRN 3275746 (2018).

\bibitem{fernandez-baleanu} A. Fernandez, D. Baleanu. Classes of Operators in Fractional Calculus: A Case Study. Math. Meth. Appl. Sci. 44(11) (2021), 9143--9162.

\bibitem{fernandez-baleanu-srivastava} A. Fernandez, D. Baleanu, H.M. Srivastava. Series representations for fractional-calculus operators involving generalised Mittag-Leffler functions. Commun. Nonlin. Sci. Numer. Simul. 67 (2019), 517--527.


\bibitem{fernandez-ozarslan-baleanu}
A. Fernandez, M.A. \"Ozarslan, D. Baleanu. On fractional calculus with general analytic kernels. Appl. Math. Comput. 354 (2019), 248--265.


\bibitem{FRS:AB} A. Fernandez, J.E. Restrepo, D. Suragan. Linear differential equations with variable coefficients and Mittag-Leffler kernels. Alex. Eng. J. 61 (2022), 4757--4763.


\bibitem{FRS} A. Fernandez, J.E. Restrepo, D. Suragan. A new representation for the solutions of fractional differential equations with variable coefficients. Under review (2020).

\bibitem{prabcap} R. Garra, R. Gorenflo, F. Polito, Z. Tomovski. Hilfer--Prabhakar derivatives and some applications. Appl. Math. Comput. 242, (2014), 576--589.

\bibitem{garrappa-maione}
R. Garrappa, G. Maione. Fractional Prabhakar Derivative and Applications in Anomalous Dielectrics: A Numerical Approach. In: A. Babiarz, A. Czornik, J. Klamka, M. Niezabitowski, eds., Theory and Applications of Non-integer Order Systems, Springer, Cham, 2017.

\bibitem{giusti-etal}
A. Giusti, I. Colombaro, R. Garra, R. Garrappa, F. Polito, M. Popolizio, F. Mainardi. A practical guide to Prabhakar fractional calculus. Fract. Calc. Appl. Anal. 23(1) (2020), 9--54.

\bibitem{mittag} R. Gorenflo, A.A. Kilbas, F. Mainardi, S.V. Rogosin. Mittag-Leffler Functions, Related Topics and Applications, 2nd ed. Springer Monographs in Mathematics, Springer, New York, 2020.

\bibitem{hilfer}
R. Hilfer, ed. Applications of Fractional Calculus in Physics. World Scientific, Singapore, 2000.

\bibitem{kilbas-marzan} A. Kilbas, S. Marzan, Cauchy problem for differential equation with Caputo derivative, Fract. Calc. Appl. Anal. 7(3) (2004), 297--321.


\bibitem{kilbasalpha} A.A. Kilbas, M. Rivero, L. Rodrignez-Germa, J.J. Trujillo, $\alpha$-Analytic solutions of some linear fractional differential equations with variable coefficients, Appl. Math. Comput. 187 (2007), 239--249.

\bibitem{generalizedfc} A.A. Kilbas, M. Saigo, R.K. Saxena. Generalized Mittag-Leffler function and generalized fractional calculus operators. Integr. Transf. Spec. F. 15(1) (2004), 31--49.

\bibitem{kilbas} A.A. Kilbas, H.M. Srivastava, J.J. Trujillo. Theory and Applications of Fractional Differential Equations. North-Holland Mathematics Studies, vol. 204. Elsevier Science B.V., Amsterdam, 2006.

\bibitem{RL} M. Kim, O. Hyong-Chol. Explicit representations of Green's function for linear fractional differential operator with variable coefficients. J. Fract. Calc. Appl. 5(1) (2014), 26--36.

\bibitem{luchko} Y. Luchko, R. Gorenflo. An operational method for solving fractional differential equations with the Caputo derivatives. Acta Math Vietnam. 24(2) (1999), 207--233.


\bibitem{miller} K.S. Miller, B. Ross. An Introduction to the Fractional Calculus and Fractional Differential Equations. John Wiley, New York, 1993.

\bibitem{oldham} K.B. Oldham, J. Spanier. The Fractional Calculus. Academic Press, New York, 1974.

\bibitem{oliveira1}
D.S. Oliveira. Properties of $\psi$-Mittag-Leffler integrals. Rendiconti del Circolo Matematico di Palermo Series 2 (2021), 1--14.

\bibitem{oliveira2}
D.S. Oliveira. $\psi$-Mittag-Leffler pseudo-fractional operators. J. Pseudo-Differ. Oper. Appl. 12(3) (2021), 1--37.

\bibitem{osler}
T.J. Osler. Leibniz rule for fractional derivatives generalized and an application to infinite series. SIAM J. Appl. Math. 18(3) (1970), 658--674.

\bibitem{oumarou-fahad-djida-fernandez}
C.M.S. Oumarou, H.M. Fahad, J.D. Djida, A. Fernandez. On fractional calculus with analytic kernels with respect to functions. Comput. Appl. Math. 40 (2021), 244.

\bibitem{analitical} S. Pak, H. Choi, K. Sin, K. Ri, Analytical solutions of linear inhomogeneous fractional differential equation with continuous variable coefficients. Adv Differ Equ 2019 (2019), 256.

\bibitem{AMS-1938} E. Pitcher, W.E. Sewel. Existence theorems for solutions of differential equations of non-integer order. Bull. Amer. Math. Soc. 44(2) (1938), 100--107.

\bibitem{podlubny} I. Podlubny. Fractional Differential Equations. Academic Press, San Diego, 1998.

\bibitem{polito} F. Polito, \u{Z}. Tomovski. Some properties of Prabhakar-type fractional calculus operators. Fractional Differ. Calc., 6(1), (2016), 73--94.

\bibitem{Prab1971} T.R. Prabhakar. A singular integral equation with a generalized Mittag-Leffler function in the kernel. Yokohama. Math. J., 19, (1971), 7--15.

\bibitem{rani-fernandez1} N. Rani, A. Fernandez. Mikusinski's operational calculus for Prabhakar fractional calculus. Int. Transf. Spec. Func. (2022), 1--21. DOI: 10.1080/10652469.2022.2057970

\bibitem{rani-fernandez2} N. Rani, A. Fernandez. Solving Prabhakar differential equations using Mikusinski’s operational calculus. Comp. Appl. Math. 41 (2022), 107.

\bibitem{RRS} J.E. Restrepo, M. Ruzhansky, D. Suragan. Explicit solutions for linear variable-coefficient fractional differential equations with respect to functions. Appl. Math. Comput. 403, (2021), 126177.

\bibitem{RRSdirac} J.E. Restrepo, M. Ruzhansky, D. Suragan. Generalized time-fractional Dirac type operators and Cauchy type problems. Under review, (2020).

\bibitem{RS:MMAS} J.E. Restrepo, D. Suragan. Oscillatory solutions of fractional integro-differential equations II. Math. Meth. Appl. Sci. 44(8) (2021), 7262--7274.

\bibitem{RSade} J.E. Restrepo, D. Suragan. Direct and inverse Cauchy problems for generalized space-time fractional differential equations. Adv. Differential Equations 26(7--8) (2021), 305--339.

\bibitem{vcapl} M. Rivero, L. Rodr\'igez--Cierm\'a, J. J. Trujillo. Linear fractional differential equations with variable coefficients, Appl. Math. Lett. 21 (2008), 892--897.

\bibitem{samko} S.G. Samko, A.A. Kilbas, O.I. Marichev. Fractional integrals and derivatives, translated from the 1987 Russian original, Gordon and Breach, Yverdon, (1993).

\bibitem{sun-etal}
H.G. Sun, Y. Zhang, D. Baleanu, W. Chen, Y.Q. Chen. A new collection of real world applications of fractional calculus in science and engineering. Commun. Nonlin. Sci. Numer. Simul. 64 (2018), 213--231.

\bibitem{tomovski-dubbeldam-korbel}
\v{Z}. Tomovski, J.L.A. Dubbeldam, J. Korbel. Applications of Hilfer--Prabhakar operator to option pricing financial model. Fract. Calc. Appl. Anal. 23(4) (2020), 996--1012.

\bibitem{zaky-hendy-suragan}
M.A. Zaky, A.S. Hendy, D. Suragan. A note on a class of Caputo fractional differential equations with respect to another function. Math. Comput. Simul. 196 (2022), 289--295.




\end{thebibliography}
\end{document}